\documentclass[a4paper,leqno]{amsart}

\usepackage{latexsym}
\usepackage[english]{babel}
\usepackage{fancyhdr}
\usepackage[mathscr]{eucal}
\usepackage{amsmath}
\usepackage{mathrsfs}
\usepackage{amsthm}
\usepackage{amsfonts}
\usepackage{amssymb}
\usepackage{amscd}
\usepackage{bbm}
\usepackage{graphicx}
\usepackage{subcaption}
\usepackage{graphics}
\usepackage{latexsym}
\usepackage{color}

\usepackage{pifont}
\usepackage{booktabs} % for '\midrule' macro

\newcommand{\ud}{\mathrm{d}}

\newcommand{\ii}{\mathrm{i}}

\newcommand{\R}{\mathbb R}

\theoremstyle{plain}
\newtheorem{theorem}{Theorem}[section]

\newtheorem{proposition}[theorem]{Proposition}

\theoremstyle{definition}

\newtheorem{remark}[theorem]{Remark}

\numberwithin{equation}{section}

\begin{document}

\title[Generalised solutions to LS/NLS with point defect]
{Generalised solutions to linear and non-linear Schr\"{o}dinger-type equations with point defect: Colombeau and non-Colombeau regimes}
\author[N.~Dugand\v{z}ija]{Nevena Dugand\v{z}ija}
\address[N.~Dugand\v{z}ija]{University of Novi Sad, Faculty of Sciences, Department of Mathematics and Informatics  \\ Trg Dositeja Obradovi\'ca 4 \\ Novi Sad (Serbia)}
\email{nevena.dugandzija@dmi.uns.ac.rs}
\author[A.~Michelangeli]{Alessandro Michelangeli}
\address[A.~Michelangeli]{Department of Mathematics and Natural Sciences, Prince Mohammad Bin Fahd University \\ Al Khobar 31952 (Saudi Arabia) \\
and Hausdorff Center for Mathematics, University of Bonn \\ Endenicher Allee 60 \\ D-53115 Bonn (Germany)\\ and TQT Trieste Institute for Theoretical Quantum Technologies, Trieste (Italy)}
\email{amichelangeli@pmu.edu.sa}
\author[I.~Vojnovi\'c]{Ivana Vojnovi\'c}
\address[I.~Vojnovi\'c - corresponding author]{University of Novi Sad \\  Faculty of Sciences\\ Department of Mathematics and Informatics  \\ Trg Dositeja Obradovi\'ca 4 \\ Novi Sad (Serbia)}
\email{ivana.vojnovic@dmi.uns.ac.rs}

%\dedicatory{}

\begin{abstract}
For a semi-linear Schr\"{o}dinger equation of Hartree type in three spatial dimensions, various approximations of singular, point-like perturbations are considered, in the form of potentials of very small range and very large magnitude, obeying different scaling limits. The corresponding nets of approximate solutions represent actual generalised solutions for the singular-perturbed Schr\"{o}dinger equation. The behaviour of such nets is investigated, comparing the distinct scaling regimes that yield, respectively, the Hartree equation with point interaction Hamiltonian vs the ordinary Hartree equation with the free Laplacian. In the second case, the distinguished regime admitting a generalised solution in the Colombeau algebra is studied, and for such a solution compatibility with the classical Hartree equation is established, in the sense of the Colombeau generalised solution theory.
\end{abstract}

\date{\today}

\subjclass[2010]{35Q55, 46F30}

\keywords{Non-linear Schr\"{o}dinger equation, Hartree equation, generalised functions and distributions, Colombeau generalised functions, compatibility of Colombeau solutions, quantum Hamiltonian of point interaction.}

\thanks{Partially supported by the Ministry of Science, Technological Development and Innovation of the Republic of Serbia (Grant No. 451-03-47/2023-01/200125)(N.D.~and I.V.), the Italian National Institute for Higher Mathematics INdAM (A.M.) and the Alexander von Humboldt Foundation, Bonn (A.M.).
}

\maketitle

%\tableofcontents

\section{Introduction and background}\label{sec:intro-main}

Popular and ubiquitous partial differential equations such as the linear Schr\"{o}dinger equation
\begin{equation}\label{eq:LS}
 \ii\partial_t u\;=\;-\Delta u + V u
\end{equation}
and the non-linear Schr\"{o}dinger equation
\begin{equation}\label{eq:NLS}
 \ii\partial_t u\;=\;-\Delta u + V u+g(u)
\end{equation}
(in the complex-valued unknown $u\equiv u(t,x)$ with $t\in\mathbb{R}$, $x\in\mathbb{R}^d$, $d\in\mathbb{N}$, and for given measurable function $V:\mathbb{R}^d\to\mathbb{R}$ and given non-linearity $g(u)$, i.e., a semi-linearity) are meaningful, among many other contexts, in applications where a point-like perturbation of the $d$-dimensional Laplacian $\Delta$ is added, modelling some type of defect or impurity extremely localised at one point or at a discrete set of points, and/or also when the initial data of the associated Cauchy problem display themselves a point-like profile of infinite magnitude.

 This class of problems appear, with different motivations and tools, in different branches of the literature.

 One mainstream concerns \emph{generalised solutions} to PDEs, beyond the ordinary solution theory in customary Sobolev spaces $H^s(\mathbb{R}^d)$, $s\geqslant 0$. This has led to fundamental developments both in the linear and in the non-linear case.

 The \emph{linear} theory of \emph{generalised functions} is the distributional one initiated by Schwartz in the late 1940's (`generalised functions' refers to the embedding of $L^1_{\mathrm{loc}}(\Omega)$ into $\mathcal{D}'(\Omega)$, for an open $\Omega\subset\mathbb{R}^d$), whose triumph for differential equations, among other fields of analysis, is best exemplified by the Ehrenpreis-Malgrange existence result 
 (see, e.g., \cite[Theorem IX.23]{rs2}, \cite[Theorem 7.3.10]{Hormander_AnalLinPartDiffOps1-1983-2003}) for which every linear PDE with constant coefficients has a solution in $\mathcal{D}'(\Omega)$), and whose limitations are represented by the existence of linear PDEs with variable (even polynomial) coefficients without solutions in $\mathcal{D}'(\Omega)$, as shown first by Lewy \cite{Lewy-1957-nosolution}. In this respect it is worth recalling that a non-$L^2$ solution theory for \eqref{eq:LS} generically features \emph{non-uniqueness}, since the Cauchy problem for \eqref{eq:LS} is a characteristic initial value problem in the sense that the initial plane $t=0$ is characteristic for the equation (we refer, e.g., to H\"{o}rmander's discussion \cite[Chapter VIII]{Hormander-LinearPDOp-1976} and \cite[Section 13.6]{Hormander_AnalLinPartDiffOps2-1983-2005}; one may also find therein an example of non-zero solutions to the free Schr\"{o}dinger equation with zero initial value -- see in particular \cite[Theorem 8.9.2 and Section 8.9]{Hormander-LinearPDOp-1976} and \cite[Example 13.6.5]{Hormander_AnalLinPartDiffOps2-1983-2005} -- however, none of such solutions belongs to $L^2(\mathbb{R}^d)$ for any positive time).

 For \emph{non-linear} PDEs a counterpart distributional solution theory is problematic in many instances -- one instructive representative is the Kenig-Ponce-Vega ill-posedness in $L^\infty([0,+\infty),\mathcal{S}'(\mathbb{R}))$ of \eqref{eq:NLS} with  $V\equiv 0$, and $g(u)=\pm|u|^2u$, when the initial datum is a Dirac delta distribution \cite{Kenig-Ponce-Vega-illpos2001}. The main underlying fundamental obstruction is Schwartz's classical result \cite{Schwartz-distrib-multipl-1954} that for an associative algebra $\mathcal{A}$ containing $\mathcal{D}'(\mathbb{R})$, having the constant function $1$ as identity, and a derivative satisfying the Leibniz rule and agreeing with the derivative in $\mathcal{D}'(\mathbb{R})$, the product in $\mathcal{A}$ cannot agree with the ordinary product on all continuous functions.

 In order to accommodate a point-like, or measure-type defect in the Cauchy problem (in the equation and/or in the initial data), and hence to manipulate, at least in the non-linear case, possibly ill-defined products of distributions, an industry has developed over the years to give meaning to generalised solutions to partial differential equations of Schr\"{o}dinger, heat, KdV, semi-linear wave, and Burgers type, among others, within suitable algebraic structures that extend the distribution space $\mathcal{D}'(\mathbb{R}^d)$ and are identified by representatives in the form of nets of smooth solutions to regularised problems, i.e., where the original singular coefficients (viewed as generalised functions) or singular initial data are suitably smoothed out. To exemplify, for concreteness with the singular version of \eqref{eq:LS} of the form $\ii\partial_t u=Au$ with ``$A=-\Delta+\delta(x)$'', a generalised solution involves the existence of limits of the form
 \begin{equation}\label{eq:A-Aeps}
  \lim_{\varepsilon\downarrow 0}\varepsilon^{-a}\big\|\ii\partial_t u_\varepsilon-A_\varepsilon u_\varepsilon\big\|_{L^\infty_t([0,T),L^2_x)}\;=\;0\,,\qquad a>0\,,
 \end{equation}
 where $(A_\varepsilon)_{\varepsilon>0}$ and $(u_\varepsilon)_{\varepsilon\in(0,1]}$ are suitable nets, respectively, of regularised operators representing $A$ and of smooth functions representing the generalised solution.

 Let us briefly outline the general line of thought behind such a construction.
 Any element $\xi\in\mathcal{D}'(\mathbb{R}^d)$ has a linear action $\varphi\mapsto\langle\xi,\varphi\rangle_{\mathcal{D}',\mathcal{D}}$ on test functions $\varphi\in\mathcal{D}(\mathbb{R}^d)$, and for a net $(\delta_\varepsilon)_{\epsilon>0}$ in $\mathcal{D}(\mathbb{R}^d)$ converging to the Dirac distribution $\delta$ one builds a family of regularisations $(\xi_\varepsilon)_{\epsilon>0}$ by convolution $\xi_\varepsilon(x)=\langle \xi(y),\delta_\varepsilon(x-y)\rangle_{\mathcal{D}',\mathcal{D}}=(\xi * \delta_\varepsilon)(x)$: indeed, $\xi_\varepsilon\to\xi$ weakly in $\mathcal{D}'(\mathbb{R}^d)$ as $\varepsilon\downarrow 0$. Identifying two nets $(\xi_\varepsilon)_{\epsilon>0}$ and $(\widetilde{\xi}_\varepsilon)_{\epsilon>0}$ by having the same limit yields a representation of $\mathcal{D}'(\mathbb{R}^d)$ as equivalent classes of nets of regularisations. In doing so, what really counts is the action of $\xi$ on nets of test functions $\mathcal{D}'$-converging to $\delta$-measures supported at points $x\in\mathbb{R}^d$, and the collection of nets of the type $(\xi_\varepsilon)_{\epsilon>0}$, with $\xi_\varepsilon\in C^\infty(\mathbb{R}^d)$, forms a differential algebra. Now, with the identification of nets \emph{with the same limit} the non-linear structure is lost, and therefore in order to have a non-linear theory of generalised functions one should manipulate regularisations but \emph{identify less}. The identification should be made by means of factorisation with respect to a differential ideal. Multiple approaches were devised following this spirit, differing by the choice of the ideal, such as those by Laugwitz \cite{Laugwitz-1961} (non-standard internal functions), Maslov \cite{Maslov-1979} (negligible sequences in asymptotics), Colombeau \cite{Colombeau-1984,Colombeau-1985} (polynomial estimates in the regularisation parameter), Egorov \cite{Egorov-1990} (localised internal functions), as well as combinations and refinements (e.g., by Rosinger \cite{Rosinger-LNM1978,Rosinger-1980}, Mallios and Rosinger \cite{Mallios-Rosinger-1999} -- see also \cite{Grosser-Farkas-Kunzinger-Steinbauer-2001,Grosser-Kunzinger-Oberguggenberger-Steinbauer-2001}).
 % aggiungere Rosinger qua 

 One major framework for this activity, among others in the same spirit, consists of the various Colombeau-type generalised function spaces and algebras, for which, beside the recap provided in Appendix \ref{app:Colombeau} here, we refer to the original constructions by Colombeau \cite{Colombeau-1984,Colombeau-1985} in the early 1980's, the precursors by Rosinger \cite{Rosinger-LNM1978,Rosinger-1980,Rosinger-1987}, the refinements by Biagioni \cite{Biagioni-LNM1990} and Oberguggenberger \cite{Oberguggenberger-1992}, as well as the modern reviews by Grosser, Farkas, Kunzinger, and Steinbauer \cite{Grosser-Farkas-Kunzinger-Steinbauer-2001}, and by Grosser, Kunzinger, Oberguggenberger, and Steinbauer \cite{Grosser-Kunzinger-Oberguggenberger-Steinbauer-2001}. Emblematic representative examples of this flourishing activity over the decades are, among others, \cite{Bu-1996,Dugandzija-Nedeljkov-2019,DV-2022-INdAMSpringer,Hormann-2011,Hormann-2018,Nedeljkov-Oberguggenberger-Pilipovic-2005,Nedeljkov-Pilipovic-1997,Nedeljkov-Pilipovic-RajterCiric-2005-heatdelta,Oberguggenberger-Wang-1994}.

 A parallel, or even preceding motivation for a linear and non-linear theory of generalised functions is to be ascribed to the need of giving rigour to heuristic computations involving free field operators (operator-valued distributions) and heuristic formulas for Hamiltonian and Lagrangian densities in quantum field theory.

 Other approaches to the study of (non-)linear Schr\"{o}dinger equations of the type \eqref{eq:LS}-\eqref{eq:NLS} with point-like defects, with methodologies and tools independent from the notion of generalised functions, are to be found in the framework of singular perturbations of elliptic operators, a mathematical subject directly motivated by and applied to quantum mechanical modelling, where the focus is to give rigorous meaning to self-adjoint realisations, over an appropriate $L^2$-space, of a differential operator such as the Laplacian $\Delta$ initially restricted to the domain of smooth functions compactly supported away from the configurations the delta-like perturbation is localised at (we refer to the general monographs \cite{albeverio-solvable,albverio-kurasov-2000-sing_pert_diff_ops,MathChallZR2021,GM-SelfAdj_book-2022}, as well as to the recent studies of specific models \cite{MP-2015-2p2,MO-2016,MO-2017,BMO-2017,MOS-2018-fractional-and-perturbation,M2020-BosonicTrimerZeroRange} and to the references therein). This line of reasoning, more operator-theoretic in nature, provides unambiguous constructions of the linear operator underlying \eqref{eq:LS}-\eqref{eq:NLS} with point-like defects, which then makes a Hilbert space solution theory possible for the induced Schr\"{o}dinger equation \cite{MOS-SingularHartree-2017,CMY-2018-2Dwaveop,CMY-2018-2DwaveopERR,MS_resonances_2020,Caccia-Finco-Noja-deltaNLS-2020,Yajima-2021-Lpbdd-delta-2D,Adami-Boni-Carlone-Tentarelli-2021,Fukaya-Georgiev-Ikeda-2021,GMS_SingularHartree2D}. A counterpart analysis in this direction can also be carried on within the formalism of non-standard analysis, in the spirit of Albeverio, Fenstad, and H\o{}egh Krohn \cite{Albeverio-Fenstad-HoeghKrohn-1979_singPert_NonstAnal} -- the recent investigation \cite{Scandone-Baglini-Simonov-2021} provides an enlightening example of this line of thought.

 The intimate implicit link between the operator-theoretic construction of self-adjoint singular-perturbed Laplacians and the above-mentioned notion of generalised solutions is represented by the circumstance that one way to recover the self-adjoint operator of interest is through a net of approximants in the form of suitably regularised Schr\"{o}dinger operators with short-scale potentials, much the same as for the approximation of ``$A=-\Delta+\delta(x)$'' by means of the net $(A_\varepsilon)_{\varepsilon>0}$ in \eqref{eq:A-Aeps} above. This idea, in the physically relevant context of the delta-type perturbation of $-\Delta$ in $L^2(\mathbb{R}^3)$, was first put forward by Berezin and Faddeev \cite{Berezin-Faddeev-1961} in the form of a renormalisation scheme for the operator $-\Delta+\mu_\varepsilon V_\varepsilon(x)$ ($V_\varepsilon$ being the potential supported on a spatial scale $\varepsilon$), and then made rigorous through a variety of techniques by Albeverio, H\o{}egh-Krohn, and Streit \cite{Albeverio-HK-Streit-1977}, Nelson \cite{Nelson-1977}, Albeverio, Fenstad, and H\o{}egh-Krohn \cite{Albeverio-Fenstad-HoeghKrohn-1979_singPert_NonstAnal}, and Albeverio and H\o{}egh-Krohn \cite{AHK-1981-JOPTH} (see also \cite{mm-2015-KP,MS-2018-shrinking-and-fractional}). Exhaustive references for general dimension $d=1,2,3$ may be found in \cite{albeverio-solvable}, and references for the related multi-particle problem may be found in \cite{M2020-BosonicTrimerZeroRange,GM-SelfAdj_book-2022}.

 The preceding overview explains the meaningfulness, say, of the class of (linear or non-linear) Cauchy problems (in the unknowns $u_\varepsilon\equiv u_\varepsilon(t,x)$)
 \begin{equation}\label{eq:epsNLS}
  \begin{cases}
   \ii\partial_t u_\varepsilon\,=\,-\Delta u_\varepsilon +V_\varepsilon u_\varepsilon+g(u_\varepsilon)\,, \\
   u_\varepsilon(0,\cdot)\,=\,a_\varepsilon
  \end{cases}\qquad (\varepsilon>0)
 \end{equation}
 associated, possibly from different conceptual perspectives, to the modification of \eqref{eq:LS} or \eqref{eq:NLS} obtained by replacing the linear operator $-\Delta$ with some unambiguous version of ``$-\Delta+\mu\delta(x)$'', where indeed $(V_\varepsilon)_{\varepsilon>0}$ and $(a_\varepsilon)_{\varepsilon>0}$ are convenient nets, respectively, of regularised potentials and regularised initial data which \emph{represent} singular, point-like objects, and the net $(u_\varepsilon)_{\varepsilon>0}$ is the \emph{representative} of a generalised solution.

 Our focus in the present study concerns an issue that therefore only becomes natural to consider in view of the above variety and overlaps of motivations and applications: the comparison of different nets $(u_\varepsilon)_{\varepsilon>0}$ of solutions to \eqref{eq:epsNLS}, depending on different scaling limits for $V_\varepsilon$ and $a_\varepsilon$, and hence the comparison between different possibilities of realising and approximating singular-perturbed NLS equations.

 The concrete set-up is presented in Section \ref{sec:setup-mainresults}, where we chose dimensionality $d=3$. An approximate singular perturbation of \eqref{eq:LS}-\eqref{eq:NLS} localised at the origin is modelled by means of potentials $V_\varepsilon$ for which a whole range of scalings in $\varepsilon$ is explored. Our main result, Theorem \ref{thm:main1} below, shows that the generic effect on the solution is negligible in the spatial $L^2$-norm of the solutions, i.e., $u_\varepsilon(t,\cdot)$ $L^2$-converges at any later time to the solution of the corresponding non-linear equation without singular perturbation. An exceptional case may occur, precisely the point interaction already well known in quantum mechanics, in a particular scaling of $V_\varepsilon$ and under special spectral requirements at zero energy. In the scaling of $V_\varepsilon$ that corresponds to the Colombeau generalised solution theory, our analysis establishes a compatibility result for a class of non-linear Schr\"odinger equations with distributional $\delta$-coefficient (Theorem \ref{thm:Colombeau-case}).

 The actual technical preparation is made in Section \ref{sec:resolvent-convergence} and consists of a control of the resolvent convergence for Schr\"{o}dinger operators of the form $-\Delta+V_\varepsilon$. Theorems \ref{thm:main1} and \ref{thm:Colombeau-case} are then proved in Section \ref{sec:proofmainresults}, and, for completeness of presentation, Appendix \ref{app:Colombeau} provides a concise review of the Colombeau algebra (and the notion of generalised solutions and compatibility for Schr\"{o}dinger equations), whereas Appendix \ref{app:Delta-alpha} summarises the definition and the main properties of the three-dimensional Hamiltonian of point interaction.

  \section{Set-up and main results}\label{sec:setup-mainresults}

  As emergent PDEs after a process of limit we consider two time-dependent Schr\"{o}dinger-type equations in three spatial dimensions, that is,
  \begin{equation}\label{eq:LS-NLS-Delta}
   \ii\partial_t u\;=\;-\Delta u +  (w*|u|^2)u
  \end{equation}
  and 
  \begin{equation}\label{eq:LS-NLS-Delta-alpha}
   \ii\partial_t v\;=\;-\Delta_\alpha v +  (w*|v|^2)v\,,
  \end{equation}
  for given $\alpha\in\mathbb{R}\cup\{\infty\}$ and given real-valued measurable potential $w$, in the unknowns $u\equiv u(t,x)$ and $v\equiv v(t,x)$, with $t\in\mathbb{R}$ and $x\in\mathbb{R}^3$, where the convolution in the cubic Hartree non-linearity is with respect to the $x$-variable. The choice $w\equiv \mathbf{0}$ vs $w \equiv\!\!\!\!\!/ \;\:\mathbf{0}$ discriminates between the linear and the non-linear problem.

  The linear action $u\mapsto\Delta_\alpha u$ in \eqref{eq:LS-NLS-Delta-alpha} refers to the singular perturbation of the Laplacian with a point interaction supported at the origin. By this one means the standard construction, concisely recalled in Appendix \ref{app:Delta-alpha}, of the self-adjoint extension, with respect to $L^2(\mathbb{R}^3)$, of the densely defined symmetric operator $(-\Delta)\upharpoonright C^\infty_c(\mathbb{R}^3\setminus\{0\})$, characterised by having inverse $s$-wave scattering length equal to $-(4\pi\alpha)^{-1}$. This is the relevant Hamiltonian of point interaction in quantum mechanics for a Schr\"{o}dinger particle coupled to the origin by means of a non-trivial interaction of zero range.

  Next, as approximants to \eqref{eq:LS-NLS-Delta} and \eqref{eq:LS-NLS-Delta-alpha}, we consider, for each $\varepsilon\in(0,1]$, the time-dependent Schr\"{o}dinger equation
    \begin{equation}\label{eq:LS-NLS-eps}
   \ii\partial_t u_\varepsilon\;=\;-\Delta u_\varepsilon + V_\varepsilon u_\varepsilon+ (w*|u_\varepsilon|^2)u_\varepsilon\,,
  \end{equation}
  in the unknown $u_\varepsilon\equiv u_\varepsilon(t,x)$, with $(t,x)\in\mathbb{R}\times\mathbb{R}^3$, where $V_\varepsilon$ is a real-valued potential effectively supported around the origin at a spatial scale $\varepsilon$, and is meant to represent a singular, delta-like profile of some magnitude and centred at $x=0$. The general law by which $V_\varepsilon$ is assumed to shrink around the origin is
  \begin{equation}\label{eq:general_scaling}
   V_\varepsilon(x)\;:=\;\frac{\,\eta(\varepsilon)\,}{\:\varepsilon^\sigma}V\Big(\frac{x}{\varepsilon}\Big)\,,
  \end{equation}
  for a given measurable function $V:\mathbb{R}^3\to\mathbb{R}$, a given smooth function $\eta:[0,1]\to[0,+\infty)$ such that $\eta(0)=\eta(1)=1$, and a given $\sigma\geqslant 0$.
   In \eqref{eq:general_scaling} the parameter $\sigma$ governs the magnitude of the scaling, whereas the function $\eta$ only provides a `distortion' factor that is relevant in certain modelling (otherwise it can be just assumed to be $\eta\equiv 1$). Convenient regularity, integrability, and decay assumptions on $V$ are needed in practice, as is going to be declared in due time.

   The scaling law \eqref{eq:general_scaling} covers distinguished and fundamentally different regimes. A meaningful one is with $\sigma=3$, in which case $V_\varepsilon\to (\int_{\mathbb{R}^3}V)\delta(x)$ in $\mathcal{D}'(\mathbb{R}^3)$ as $\varepsilon\downarrow 0$, provided that $\int_{\mathbb{R}^3}V$ is (non-zero, and) finite. In particular, as discussed in Appendix \ref{app:Colombeau}, the case where $V\in C^\infty_c(\mathbb{R}^3)$ or $V\in\mathcal{S}(\mathbb{R}^3)$, $V\geqslant 0$ (and not identically zero), $\eta\equiv 1$, and $\sigma=3$ is the standard `\emph{Colombeau regime}', the net $(u_\varepsilon)_{\varepsilon\in(0,1]}$ then representing a generalised solution to the formal equation
    \begin{equation}\label{eq:LS-NLS-Delta-delta}
   \ii\partial_t u\;=\;-\Delta u + a_V\delta u +  (w*|u|^2)u\,,
  \end{equation}
  where $a_V:=\int_{\mathbb{R}^3}V(x)\,\ud x$.

  The regimes $\sigma\in[0,3)$ and $\sigma>3$ in \eqref{eq:general_scaling} represent a point-like defect that, as compared to \eqref{eq:LS-NLS-Delta-delta}, is, respectively, \emph{much weaker} or \emph{much stronger} than the Dirac delta. Somewhat counter-intuitively, as we shall see, the net effect may be a trivial or non-trivial ``impurity'' localised at $x=0$. In fact, we shall prove that as long as $\sigma\in[0,3]$ the only non-trivial interaction that may emerge in the limit $\varepsilon\downarrow 0$ is in the already well known distinguished regime $\sigma=2$, under the spectral condition of presence of a zero-energy resonance for $-\Delta+V$, and absence of zero-energy eigenvalue, on which we shall comment in a moment.

  It is important to underline for the present purposes the following well-posedness result under standard working assumptions on the potentials $V$ and $w$.

  \begin{theorem}\label{thm:well-posedness}
   When
   \begin{equation}\label{eq:conditions-w}
    w\,\in\,L^\infty(\mathbb{R}^3,\mathbb{R})\cap W^{1,3}(\mathbb{R}^3,\mathbb{R})\,,\qquad \textrm{$w$ is an even function,}
   \end{equation}
   and 
  \begin{equation}\label{eq:conditions-V-L32p-Linf}
   %\textrm{$V$ is real-valued}\,,\qquad 
   V\in L^p(\mathbb{R}^3,\mathbb{R})+L^\infty(\mathbb{R}^3,\mathbb{R})\,,\qquad p>\frac{3}{2}\,,%\mathcal{R} %\cap L^1(\mathbb{R}^3,\mathbb{R})\,,
  \end{equation}
   all three equations \eqref{eq:LS-NLS-Delta}-\eqref{eq:LS-NLS-eps} are globally well-posed in $L^2(\mathbb{R}^3)$.    
   In particular, for given initial datum in $L^2(\mathbb{R}^3)$ at $t=0$ they all admit a unique strong $L^2$-solution preserving the $L^2$-norm at all times.
  \end{theorem}

   Here $W^{s,p}$ is the standard Sobolev space notation. Theorem \ref{thm:well-posedness} is completely classical  for the ordinary (non-linear) Schr\"{o}dinger equations \eqref{eq:LS-NLS-Delta} and \eqref{eq:LS-NLS-eps}, ({\color{red} in fact, it is formulated here, for the purposes of the present discussion, with even more restrictive hypotheses than the usual ones}: see, e.g., \cite[Theorem 4.6.1 and Corollary 4.6.5]{cazenave}), and it was recently established for the first time in \cite[Theorem 1.5]{MOS-SingularHartree-2017} for the singular-perturbed equation \eqref{eq:LS-NLS-Delta-alpha}.

   \begin{remark} If $V$ satisfying \eqref{eq:conditions-V-L32p-Linf} has a local $L^p$-singularity for some $p>\frac{3}{2}$, then that is also a $L^{\frac{3}{2}}$-singularity. Moreover, $L^{\frac{3}{2}}(\mathbb{R}^3,\mathbb{R})\subset\mathcal{R}$,  where $\mathcal{R}$ is the Rollnik class over $\mathbb{R}^3$, consisting of the functions $f$ for which the Rollnik norm 
  \begin{equation}
   \|f\|_{\mathcal{R}}\;:=\;\bigg(\iint_{\mathbb{R}^3\times\mathbb{R}^3}\frac{\,\overline{f(x)}\,f(y)}{|x-y|^2}\,\ud x\,\ud y\bigg)^{\!\frac{1}{2}}
  \end{equation}
  is finite. For potentials $V\in \mathcal{R}+L^\infty(\mathbb{R}^3,\mathbb{R})$ the Schr\"{o}dinger operator $-\Delta+V$ is self-adjointly realised in $L^2(\mathbb{R}^3)$ as a form sum with quadratic form $H^1(\mathbb{R}^3)$ \cite[Theorems X.17 and X.19]{rs2}.
   \end{remark}

  Next, let us formulate the above-mentioned resonance condition explicitly. It is actually a well understood regime in quantum mechanics, which is known to produce, in the limit $\varepsilon\downarrow 0$, a point interaction of the type $-\Delta_\alpha$. When
  \begin{equation}\label{eq:conditions-V-Rollnik}
   \textrm{$V$ is real-valued}\,,\qquad 
   V\in \mathcal{R}\,, %\cap L^1(\mathbb{R}^3,\mathbb{R})
  \end{equation}
  %the Schr\"{o}dinger operator $-\Delta+V$ is self-adjointly realised in $L^2(\mathbb{R}^3)$ as a form sum with quadratic form $H^1(\mathbb{R}^3)$ \cite[Theorems X.17 and X.19]{rs2}, and moreover 
  the Birman-Schwinger operator $u(-\Delta)^{-1}v$, where $v(x):=\sqrt{|V(x)|}$ and $u(x):=\sqrt{|V(x)|\,}\,\mathrm{sign}(V(x))$, is compact in $L^2(\mathbb{R}^3)$ (it is actually a Hilbert-Schmidt operator, as is straightforward to check). When, in addition to \eqref{eq:conditions-V-Rollnik},
  \begin{equation}\label{eq:conditions-V-L1}
   V\in L^1(\mathbb{R}^3,\mathbb{R})\,,
  \end{equation}
  one says that $-\Delta+V$ satisfies the \emph{(simple) resonance condition} at zero energy if $u(-\Delta)^{-1}v$ admits the simple eigenvalue $-1$, that is,
  \begin{equation}\label{eq:ZERcondphi}
      \begin{array}{c}
    \exists\,\phi\in L^2(\mathbb{R}^3)\setminus\{\mathbf{0}\} \\
    \textrm{unique, up to multiples}
   \end{array}\qquad 
\textrm{ such that }\qquad  u(-\Delta)^{-1}v\,\phi\;=\;-\phi
  \end{equation}
(in fact such a $\phi$ can be chosen to be real-valued), and if \emph{in addition} the function
 \begin{equation}\label{eq:psi-resonance}
  \psi\;:=\;(-\Delta)^{-1}v\phi
 \end{equation}
 satisfies
 \begin{equation}\label{eq:ZERcondphi2}
  \psi\in L^2_{\mathrm{loc}}(\mathbb{R}^3)\setminus L^2(\mathbb{R}^3)\,.
 \end{equation}
 Actually (see, e.g., \cite[Lemma I.1.2.3]{albeverio-solvable}), for any non-zero $\phi\in L^2(\mathbb{R}^3)$ satisfying $u(-\Delta)^{-1}v\,\phi=-\phi$ (under \eqref{eq:conditions-V-Rollnik}-\eqref{eq:conditions-V-L1}) the corresponding $\psi:=(-\Delta)^{-1}v\phi$ has the properties $\psi\in L^2_{\mathrm{loc}}(\mathbb{R}^3)$, $(-\Delta+V)\psi=0$ in $\mathcal{D}'(\mathbb{R}^3)$, and
 \begin{equation}\label{eq:psiL2locnotL2}
  \psi\in L^2_{\mathrm{loc}}(\mathbb{R}^3)\setminus L^2(\mathbb{R}^3)\qquad\Leftrightarrow\qquad \int_{\mathbb{R}^3}v\phi\,\ud x\,=\,\int_{\mathbb{R}^3} V\psi\,\ud x\,\neq\,0\,.
 \end{equation}
  Under the resonance condition \eqref{eq:conditions-V-Rollnik}-\eqref{eq:conditions-V-L1}-\eqref{eq:ZERcondphi}-\eqref{eq:ZERcondphi2} the function $\psi$ is referred to as the \emph{zero-energy resonance} for $-\Delta+V$. Furthermore, \eqref{eq:conditions-V-Rollnik}-\eqref{eq:conditions-V-L1}-\eqref{eq:ZERcondphi}-\eqref{eq:ZERcondphi2} provide a condition of \emph{lack of zero-energy eigenvalue} for $-\Delta+V$: for, if $(-\Delta+V)\psi=0$ for some $\psi\in H^1(\mathbb{R}^3)$, then $\phi:=u\psi\in L^2(\mathbb{R}^3)\setminus\{0\}$ (otherwise, $-\Delta\psi=-v\phi=0$, which is impossible), and $u(-\Delta)^{-1}v\phi=u(-\Delta)^{-1}V\psi=-u\psi=-\phi$, but by assumption there is only one such $\phi$ (up to multiples) and the corresponding $\psi$ does not belong to $L^2(\mathbb{R}^3)$.

%   For our various needs (well-posedness of the PDE of interest, self-adjoint realisation of the Schr\"{o}dinger operator of interest, and control of the limit $\varepsilon\downarrow 0$)
%   \begin{equation}
%    \textrm{$V$ is real-valued}\,,\qquad V\in L^1(\mathbb{R}^3)\cap L^p(\mathbb{R}^3)\cap \mathcal{R}
%   \end{equation}

{\color{red}  In fact the above assumptions, including the restriction to the Rollnik class, may be further relaxed while keeping the following discussion essentially unaltered; yet, the main focus here is the magnitude of the scaling, and not the generality of the un-scaled potential.}

  Our first main result reads as follows.

  \begin{theorem}\label{thm:main1}
  The following are given: $\mathsf{T}>0$, $w$ satisfying \eqref{eq:conditions-w}, $V$ satisfying \eqref{eq:conditions-V-Rollnik}-\eqref{eq:conditions-V-L1}, and additionally \eqref{eq:conditions-V-L32p-Linf} when $w \equiv\!\!\!\!\!/ \;\:\mathbf{0}$, $\eta:[0,1]\to[0,+\infty)$ such that it is smooth and $\eta(0)=\eta(1)=1$, and $a\in L^2(\mathbb{R}^3)$. For fixed $\sigma>0$ and $\varepsilon\in(0,1]$, define $V_\varepsilon$ as in \eqref{eq:general_scaling}. Correspondingly, let $u_\varepsilon$ be the unique solution in $C(\mathbb{R},L^2(\mathbb{R}^3))$ to the Cauchy problem associated with \eqref{eq:LS-NLS-eps} with initial datum $a$.
   \begin{enumerate}
    \item[(i)] \emph{\textbf{[Colombeau or sub-Colombeau regimes, non-resonant.]}} Let $\sigma\in[0,3]$. In the special case $\sigma=2$, assume in addition that $-1$ is \emph{not} an eigenvalue of $u(-\Delta)^{-1}v$. When $\sigma\in(2,3]$ assume that $V$, if not identically zero, is non-negative. Then,
    \begin{equation}\label{eq:mainthmlimit1}
     \lim_{\varepsilon\downarrow 0}\|u_\varepsilon-u\|_{L^\infty([0,\mathsf{T}],L^2(\mathbb{R}^3))}=\;0\,,
     %\lim_{\varepsilon\downarrow 0}\|u_\varepsilon-u\|_{L^\infty(\mathbb{R}_t,L^2_x)}\;=\;0\,,
    \end{equation}
   where $u$ the unique solution in $C(\mathbb{R},L^2(\mathbb{R}^3))$ to the Cauchy problem
     associated with \eqref{eq:LS-NLS-Delta} with initial datum $a$.
   \item[(ii)] \emph{\textbf{[Resonant point interaction case.]}} Let $\sigma=2$ and assume in addition that $V$ satisfies the resonance condition \eqref{eq:ZERcondphi}-\eqref{eq:ZERcondphi2} with normalisation   \begin{equation}\label{eq:phinormalisation}
    \int_{\mathbb{R}^3}\mathrm{sign}(V)|\phi|^2\,\ud x\:=\:-1\,.
   \end{equation}
   Then,
    \begin{equation}\label{eq:mainthmlimit2}
    \lim_{\varepsilon\downarrow 0}\|u_\varepsilon-v\|_{L^\infty([0,\mathsf{T}],L^2(\mathbb{R}^3))}=\;0\,,
    \end{equation}
   where $v$ the unique solution in $C(\mathbb{R},L^2(\mathbb{R}^3))$ to the Cauchy problem
     associated with \eqref{eq:LS-NLS-Delta-alpha} with initial datum $a$, with 
     \begin{equation}
      \alpha\;:=\;-\frac{\eta'(0)}{\big|\int_{\mathbb{R}^3} V\psi\,\ud x\,\big|^2}\,.
     \end{equation}
    %and $H^2_\alpha(\mathbb{R}^3):=\mathrm{dom}(-\Delta_\alpha)$.
   \end{enumerate}
  \end{theorem}

  Theorem \ref{thm:main1} conveys multiple information, in particular on the comparison between the distinguished regimes $\sigma=2$ with resonance, and $\sigma=3$.
  
  The end-point regime $\sigma=3$ is relevant in the framework of the Colombeau generalised solution theory for non-linear Schr\"{o}dinger equations (Appendix \ref{app:Colombeau}). 
%   That is the framework requiring:
%   \begin{itemize}
%    \item $\eta\equiv 1$, 
%    \item a rapresentative net $(V_\varepsilon)_{\varepsilon\in(0,1]}$ of distributional approximants of $\delta(x)$ consising, without loss of generality, of non-negative Schwartz functions (i.e., smooth and with rapid decrease),
%    \item a net $(a_\varepsilon)_{\varepsilon\in(0,1]}$ of initial data of the form $\rho_\varepsilon*a$ for some (smooth and compactly supported) $\delta$-mollifier $\rho_\varepsilon$.
%   \end{itemize}
%  
 The control of such a regime can be deduced as a corollary from Theorem \ref{thm:main1} and combined with the Colombeau solution theory for \eqref{eq:LS-NLS-Delta-delta}, a non-linear Schr\"{o}dinger equation with distributional coefficient in its linear part.
 %, for which existence and uniqueness in the $H^2$-based Colombeau algebra has been already established in \cite{Dugandzija-Vojnovic-2021}. 
 In particular, Theorem \ref{thm:main1} can be applied to deduce \emph{compatibility}, in the precise meaning of the Colombeau solution theory (see Appendix \ref{app:Colombeau} for details), between the Colombeau solution to \eqref{eq:LS-NLS-Delta-delta}, namely the equivalence class of nets of the form $(u_\varepsilon)_{\varepsilon\in(0,1]}$ modulo negligible nets (in the sense of \eqref{eq:Gquotient}) and the solution $u$ to the non-linear Schr\"{o}dinger equation \eqref{eq:LS-NLS-Delta}, in which the linear part is simply given by the free Laplacian. For the relevance of this case within the generalised solution theory of non-linear PDEs, it is instructive to cast it in a separate theorem.

 \begin{theorem}\label{thm:Colombeau-case}
  Let $a\in L^2(\mathbb{R}^3)$ and let $w$ be a function satisfying \eqref{eq:conditions-w}. The Colombeau generalised solution to the Cauchy problem for the singular non-linear Schr\"{o}dinger equation
   \begin{equation}\label{eq:deltaeq}
     \ii\partial_t u\;=\;-\Delta u + \delta u +  (w*|u|^2)u
   \end{equation}
   with initial datum $a$ is compatible with the $H^2$-solution to the Cauchy problem for the classical non-linear Schr\"{o}dinger equation
      \begin{equation}\label{eq:nodeltaeq}
     \ii\partial_t u\;=\;-\Delta u + (w*|u|^2)u
   \end{equation}
  with the same initial datum.
%   
%   and $\phi\in \mathcal{S}(\mathbb{R}^3)$ with $\phi> 0$ almost everywhere and $\int_{\mathbb{R}^3}\phi=1$. Let $w$ be a function satisfying \eqref{eq:conditions-w}. Let $\varepsilon\in(0,1]$ and for each such $\varepsilon$ let   
%  \begin{equation}
%   a_\varepsilon:=\rho_\varepsilon* a\,,\qquad \phi_\varepsilon(x)\;:=\;\frac{1}{\,\varepsilon^3}\phi\Big(\frac{x}{\varepsilon}\Big)\,,\qquad x\in\mathbb{R}\,,
%   \end{equation}
%    where $(\rho_\varepsilon)_{\varepsilon>0}$ is a net of positive, smooth, and compactly supported $\delta$-mollifiers. Then the Cauchy problem for the singular non-linear Schr\"{o}dinger equation
%    \begin{equation}
%      \ii\partial_t u\;=\;-\Delta u + \delta u +  (w*|u|^2)u
%    \end{equation}
%   admits a unique Colombeau solution which is compatible with the solution to the corresponding Cauchy problem
%   
 \end{theorem}

%  As reviewed in Appendix \ref{app:Colombeau}, the $H^2$-based Colombeau generalised solution to \eqref{eq:deltaeq} is actually a net $(u_\varepsilon)_{\varepsilon\in(0,1]}$ of functions that are continuous in time with values in $H^2(\mathbb{R}^3)$ (hence, in particular, in $L^2(\mathbb{R}^3)$) and satisfy
%  
%  Upon embedding $a$ into the Colombeau algebra, as an element $(a_\varepsilon)_{\varepsilon\in(0,1]}$, where $a_\varepsilon:= a*\rho_\varepsilon$ and $(\rho_\varepsilon)_{\varepsilon>0}$ is a net of positive, smooth, and compactly supported $\delta$-mollifiers, the

\begin{remark}
 In \cite{Dugandzija-Vojnovic-2021} existence and uniqueness of a $H^2$-based Colombeau solution of $\eqref{eq:deltaeq}$ was established with slightly different conditions on the convolution potential $w$ (more precisely, for even-symmetric $w\in W^{2,p}(\mathbb{R}^3,\mathbb{R})$, $p\in(2,\infty]$). In addition compatibility was examined, but only for odd-symmetric $a\in H^2(\mathbb{R}^3)$ an in the case of classical equation $\eqref{eq:nodeltaeq}$. Besides, the embedding 
 (in the sense of Appendix \ref{app:Colombeau} and \eqref{eq:colemb} in particular) of the $\delta$-coefficient appearing in \eqref{eq:deltaeq} into the Colombeau algebra was performed in \cite{Dugandzija-Vojnovic-2021} by means of \emph{compactly supported} mollifiers, whereas, without altering the general picture of the Colombeau generalised solution theory, our Theorem \ref{thm:Colombeau-case} requires the mollifiers to be \emph{strictly positive} almost everywhere. In this respect, Theorem \ref{thm:Colombeau-case} supplements the recent analysis made in \cite{Dugandzija-Vojnovic-2021} for the Hartree equation, as well as the related works \cite{Dugandzija-Nedeljkov-2019,DV-2022-INdAMSpringer} for the cubic non-linear Schr\"{o}dinger equation.
\end{remark}

 All \emph{sub-Colombeau} regimes $\sigma\in[0,3)$, but the resonant $\sigma=2$ case, display the same qualitative `compatibility' behaviour expressed by \eqref{eq:mainthmlimit1}. Thus, notably, the perturbation $V_\varepsilon$ in \eqref{eq:LS-NLS-eps}, localised at the very small spatial scale $\varepsilon$ and with very high magnitude $\varepsilon^{-\sigma}$, is practically ineffective for the solution, in the $L^2$-sense. The only non-trivial effect, well known in the quantum mechanical context, is produced by the perturbation  when $\sigma=2$, in the exceptional resonant case, namely under exceptional spectral requirements at zero energy. In this sense, one virtue of Theorem \ref{thm:main1}, which was in fact a primary motivation for us, is the clarification of the different quantitative meaning (depending on the different scaling) of linear and semi-linear Schr\"{o}dinger equations with `approximated point-like singularity'.

 As for the \emph{super-Colombeau} regimes $\sigma>3$, our scheme to establish Theorem \ref{thm:res-convergence}, which is the actual technical basis Theorem \ref{thm:main1} is built upon, irremediably breaks down. In Remark \ref{rem:sigma-above3} we elaborate on the strategy that is naturally expected to be needed in order to analyse those regimes.

 We should like to conclude this presentation of the main results by mentioning another potential virtue of an analysis of this kind, or at least of its spirit. Precisely because the limiting behaviour of solutions to non-linear (Schr\"{o}dinger) equations with approximate point-like perturbations appears to be a relevant information in several contexts, also quite independent in their motivations, it is of interest to investigate the general issue of how the \emph{dispersive properties} of the linear propagator behave as $\varepsilon\downarrow 0$. The latter are, as well known, a fundamental tool for the solution theory at fixed $\varepsilon>0$, and a control of them for each $\varepsilon$ would allow to export \emph{after} the limit $\varepsilon\downarrow 0$ the knowledge concerning the dispersive and scattering properties, the growth of Sobolev norms, etc., of the approximate solutions. In practice this would require controls of limits such as \eqref{eq:mainthmlimit1}-\eqref{eq:mainthmlimit2}, but in higher regularity norms and weighted norms. A very recent discussion of this essentially uncharted problem, with a few preliminary answers, is provided in \cite{GMS2022_epsilon-dispersive}.

  \section{Resolvent convergence}\label{sec:resolvent-convergence}
  
  In this Section we shall establish the main technical result of the present work (Theorem \ref{thm:res-convergence}). It provides distinct resolvent convergences in the two regimes of our main theorem (Theorem \ref{thm:main1}).

  \begin{theorem}\label{thm:res-convergence}
   For each $\varepsilon\in(0,1]$ let $H_\varepsilon:=-\Delta+V_\varepsilon$ be the lower semi-bounded self-adjoint operator in $L^2(\mathbb{R}^3)$ defined as form sum with (energy) quadratic form domain $H^1(\mathbb{R}^3)$, where $V_\varepsilon$ is defined by \eqref{eq:general_scaling} for a given $V\in\mathcal{R}\cap L^1(\mathbb{R}^3)$, a given $\sigma>0$, and a given smooth $\eta:[0,1]\to[0,+\infty)$ satisfying $\eta(0)=\eta(1)=1$. 
   %Let $\lambda>0$ with $-\lambda\notin\sigma(H_\varepsilon)$ for each (sufficiently small) $\varepsilon$.
   \begin{enumerate}
    \item[(i)] Let $\sigma\in[0,3]$. In addition, when $\sigma=2$ assume that $-1$ is \emph{not} an eigenvalue of $u(-\Delta)^{-1}v$, and when $\sigma\in(2,3]$ assume that $V>0$ almost everywhere. Then
    \begin{equation}\label{eq:resolventlimit1}
      H_\varepsilon\;\xrightarrow{\;\varepsilon\downarrow 0 \;} \;-\Delta
    \end{equation}
    in the \emph{norm resolvent sense} when $\sigma\in[0,2]$ and in \emph{strong resolvent sense} when $\sigma\in(2,3]$, where $-\Delta$ above is the self-adjoint negative Laplacian in $L^2(\mathbb{R}^3)$ with domain $H^2(\mathbb{R}^3)$.
    % in the Hilbert space operator norm when $\sigma\in[0,2]$, and in the strong Hilbert space operator topology when $\sigma\in(2,3]$.    
    \item[(ii)] Let $\sigma=2$ and moreover assume that $(1+|\cdot|)V\in L^1(\mathbb{R}^3)$ and that $V$ satisfies the resonance condition \eqref{eq:ZERcondphi}-\eqref{eq:ZERcondphi2}. Let $\psi\in L^2_{\mathrm{loc}}(\mathbb{R}^3)\setminus L^2(\mathbb{R}^3)$ be the corresponding zero-energy resonance function \eqref{eq:psi-resonance} with normalisation \eqref{eq:phinormalisation}, and set 
         \begin{equation}\label{eq:thm2-def-alpha}
      \alpha\;:=\;-\frac{\eta'(0)}{\big|\int_{\mathbb{R}^3} V\psi\,\ud x\,\big|^2}\,.
     \end{equation}
    Then 
        \begin{equation}\label{eq:resolventlimit2}
      H_\varepsilon\;\xrightarrow{\;\varepsilon\downarrow 0 \;} \;-\Delta_\alpha
    \end{equation}
    in the \emph{norm resolvent sense}, where $-\Delta_\alpha$ is the self-adjoint Hamiltonian of point interaction in $L^2(\mathbb{R}^3)$ supported at the origin and with $s$-wave scattering length $-(4\pi\alpha)^{-1}$, whose definition is recalled in \eqref{eq:defDelta-alpha}.
   \end{enumerate}
  \end{theorem}

  Theory and properties of the resolvent convergence are discussed, e.g., in \cite[Section VIII.7]{rs1}.
  In practice the limits \eqref{eq:resolventlimit1}-\eqref{eq:resolventlimit2} have the form $(H_\varepsilon+\lambda\mathbbm{1})^{-1}\to(-\Delta+\lambda\mathbbm{1})^{-1}$ or $(H_\varepsilon+\lambda\mathbbm{1})^{-1}\to(-\Delta_\alpha+\lambda\mathbbm{1})^{-1}$ in the strong or norm Hilbert space operator topology for all $\lambda\in\mathbb{C}\setminus\mathbb{R}$, as well as for all $\lambda>0$ with $-\lambda\notin\sigma(H_\varepsilon)$ for each (sufficiently small) $\varepsilon$.

  The case $\sigma=2$ is already well known in the literature and it is included in Theorem \ref{thm:res-convergence} for completeness of presentation. It is central in quantum mechanics, in the context of approximation of the point interaction Hamiltonian by means of Schr\"{o}dinger operators with regular potentials on a shorter and shorter spatial scale (see, e.g., \cite[Section I.1.2]{albeverio-solvable}). Notably, scaling $V_\varepsilon$ as in \eqref{eq:general_scaling} with other powers below and above the threshold $\sigma=2$ seems to have only attracted marginal attention in that context (as opposed to the \emph{one-dimensional} case, for which a number of studies have been made on scaling regimes that are more singular than the scaling yielding the one-dimensional point interaction: \cite{Seba1985-halfline,Seba1986-deltaprime,Seba1986-regularised1D,Golovaty-IEOT2013,DM-2015-halfline,Golovaty-IEOT2018}). Certain meanings for quantum systems of particles have been recently attributed to the three-dimensional scaling with $\sigma=3$ in \cite{DellAntonio-contact-2019}. 
  In a related context of Maxwell-Lorentz systems with dipole approximation, an analogous version of the limit in the $\sigma=3$ case of Theorem \ref{thm:res-convergence} above was previously studied in \cite{Noja-Posilicano_AHP1998,Noja-Posilicano_AHP1999} in the \emph{special case} where the charge density of the electron $\rho_\varepsilon(x)$ (the analogue of $V_\varepsilon$ here) is implemented as a \emph{separable potential} $|\rho_\varepsilon\rangle\langle\rho_\varepsilon|$, a caricature model that makes all computations more explicit and accessible: scaling $\rho_\varepsilon(x)$ indeed as $\varepsilon^{-3}\rho(x/\varepsilon)$ is shown to produce a trivial limit unless a suitable regularisation (of the electron's mass) is performed.

  The preparation of the proof of Theorem \ref{thm:res-convergence} requires Propositions \ref{prop:AFC}, \ref{prop:AClimits}, and \ref{prop:Fepslimit} below.   
  For $\lambda>0$ we shall set (see \eqref{eq:GlambdaApp})
  \begin{equation}\label{eq:Glambda}
  G_\lambda(x)\;:=\;\frac{\,e^{-|x|\sqrt{\lambda}}\,}{4\pi|x|}\,.
 \end{equation}
  Recall also the short hand $|f\rangle\langle g|$, given any two $f,g\in L^2(\mathbb{R}^3)$, for the $L^2\to L^2$ rank one operator $h\mapsto(\int_{\mathbb{R}^3}\overline{g} h)f$.

  \begin{proposition}\label{prop:AFC}
   For each $\varepsilon\in(0,1]$ let $H_\varepsilon:=-\Delta+V_\varepsilon$ be the lower semi-bounded self-adjoint operator in $L^2(\mathbb{R}^3)$ defined as form sum with (energy) quadratic form domain $H^1(\mathbb{R}^3)$, where $V_\varepsilon$ is defined by \eqref{eq:general_scaling} for a given $V\in\mathcal{R}\cap L^1(\mathbb{R}^3)$, a given $\sigma\geqslant 0$, and a given smooth $\eta:[0,1]\to[0,+\infty)$ satisfying $\eta(0)=\eta(1)=1$.
   Let $\lambda>0$ with $-\lambda\notin\sigma(H_\varepsilon)$. Then,
   \begin{equation}\label{eq:AFCeps}
    (H_\varepsilon+\lambda\mathbbm{1})^{-1}-(-\Delta+\lambda\mathbbm{1})^{-1}\;=\;-A_\varepsilon^{(\lambda)} F_\varepsilon^{(\lambda)} C_\varepsilon^{(\lambda)}
   \end{equation}
   as an identity between everywhere defined and bounded operators on $L^2(\mathbb{R}^3)$, where $A_\varepsilon^{(\lambda)}$ and $C_\varepsilon^{(\lambda)}$ are the Hilbert-Schmidt operators with integral kernel, respectively,
   \begin{equation}\label{eq:defAepsCeps}
    \begin{split}
     A_\varepsilon^{(\lambda)}(x,y)\;&:=\;G_\lambda(x-\varepsilon y)\sqrt{|V(y)|}\,, \\
     C_\varepsilon^{(\lambda)}(x,y)\;&:=\;\mathrm{sign}(V(x))\sqrt{|V(x)|}\,G_\lambda(\varepsilon x- y)\,,
    \end{split}
   \end{equation}
   and 
   \begin{equation}\label{eq:defFeps}
    F_\varepsilon^{(\lambda)}\;:=\;\varepsilon\,\eta(\varepsilon)\Big(\varepsilon^{\sigma-2}\mathbbm{1}+\eta(\varepsilon)\,\mathrm{sign}(V)\sqrt{|V|}(-\Delta+\varepsilon^2\lambda\mathbbm{1})^{-1}\sqrt{|V|}\Big)^{-1}\,.
   \end{equation}
  \end{proposition}

  \begin{proof}
   Set $v_\varepsilon:=\sqrt{|V_\varepsilon|}$ and $u_\varepsilon:=\mathrm{sign}(V_\varepsilon)\sqrt{|V_\varepsilon|}$, as well as $v:=v_1$ and $u:=u_1$. Clearly, $u_\varepsilon v_\varepsilon=V_\varepsilon$ and $uv=V$. As anticipated when \eqref{eq:ZERcondphi} was formulated, the operator $v_\varepsilon(-\Delta-k^2\mathbbm{1})^{-1}v_\varepsilon$ is compact on $L^2(\mathbb{R}^3)$ for any $k\in\mathbb{C}$ with $\mathfrak{Im}\,k>0$ and $k^2\notin\sigma(H_\varepsilon)$, with Hilbert-Schmidt norm
   \[
    \begin{split}
         \big\|v_\varepsilon(-\Delta-k^2\mathbbm{1})^{-1}v_\varepsilon\big\|_{\mathrm{H.S.}}^2\;&=\;\frac{1}{\:(4\pi)^2}\iint_{\mathbb{R}^3\times\mathbb{R}^3}\ud x\,\ud y\,\frac{\,|V_\varepsilon(x)|\,|V_\varepsilon(y)|\,}{|x-y|^2}\,e^{-2|x-y| \mathfrak{Im}\,k} \\
    &\leqslant\;\varepsilon^{4-2\sigma}\|V\|_{\mathcal{R}}^2\,.
    \end{split}
   \]
   As a consequence, the Hilbert-Schmidt-operator-valued map $k\mapsto v_\varepsilon(-\Delta-k^2\mathbbm{1})^{-1}v_\varepsilon$ is analytic on the half-plane $\mathfrak{Im}\,k>0$ and continuous for $\mathfrak{Im}\,k\geqslant 0$ (and $k^2\notin\sigma(H_\varepsilon)$), and by dominated convergence
   \[
    \big\|v_\varepsilon(-\Delta-k^2\mathbbm{1})^{-1}v_\varepsilon\big\|_{\mathrm{H.S.}}\;\xrightarrow[\;\mathfrak{Im}\,k>0\,,\; k^2\notin\sigma(H_\varepsilon)\;]{|k|\to +\infty}\;0\,.
   \]
  The latter limit implies that for arbitrary $\delta>0$ there is $\lambda>0$ large enough and hence $k=\ii\sqrt{\lambda}$ such that 
  \[
   \big\|(-\Delta-k^2\mathbbm{1})^{-\frac{1}{2}}\,|V_\varepsilon|\,(-\Delta-k^2\mathbbm{1})^{-\frac{1}{2}}\big\|_{\mathrm{H.S.}}\;=\;\big\|v_\varepsilon(-\Delta+\lambda\mathbbm{1})^{-1}v_\varepsilon\big\|_{\mathrm{H.S.}}\;\leqslant\;\delta\,,
  \]
  whence also
  \[
   |\langle\varphi,V_\varepsilon\varphi\rangle_{L^2}|\;\leqslant\;\delta\langle\varphi,(-\Delta)\varphi\rangle_{L^2}+a_\delta\|\varphi\|_{L^2}^2\qquad\forall\psi\in H^1(\mathbb{R}^3)
  \]
  for some $a_\delta>0$. This shows that the operator of multiplication by $\sqrt{|V_\varepsilon|}$, and therefore the operators of multiplication by $u_\varepsilon$ and by $v_\varepsilon$, are all infinitesimally small with respect to $(-\Delta)^{\frac{1}{2}}$. The latter condition, together with the compactness of 
   $u_\varepsilon(-\Delta-k^2\mathbbm{1})^{-1}v_\varepsilon$ when $\mathfrak{Im}\,k>0$ and $k^2\notin\sigma(H_\varepsilon)$, allow for the applicability of the Konno-Kuroda resolvent formula (see, e.g., \cite[Theorem B.1(b)]{albeverio-solvable}), thus yielding
   \[
    \begin{split}
     & (H_\varepsilon+\lambda\mathbbm{1})^{-1}-(-\Delta+\lambda\mathbbm{1})^{-1} \\
     & \qquad =\;-(-\Delta+\lambda\mathbbm{1})^{-1} v_\varepsilon\big(\mathbbm{1}+u_\varepsilon(-\Delta+\lambda\mathbbm{1})^{-1}v_\varepsilon\big)^{-1}u_\varepsilon(-\Delta+\lambda\mathbbm{1})^{-1}
    \end{split}
   \]
   for any $\lambda>0$ with $-\lambda\notin\sigma(H_\varepsilon)$.
  Inserting $\mathbbm{1}=U_\varepsilon U_\varepsilon^*$, where $U_\varepsilon$ is the unitary operator of dilation in $L^2(\mathbb{R}^3)$ defined by $(U_\varepsilon f)(x):=\varepsilon^{-3/2}f(x/\varepsilon)$, and exploiting \eqref{eq:general_scaling} and the scaling properties
  \[
   U_\varepsilon^*M(x)U_\varepsilon\,=\,M(\varepsilon x)\,,\qquad U_\varepsilon^*\Delta U_\varepsilon\,=\,\varepsilon^{-2}\Delta
  \]
  (meaning $M(x)$ and $M(\varepsilon x)$ above as multiplication operators), one finds
  \[
   \begin{split}
    & (-\Delta+\lambda\mathbbm{1})^{-1} v_\varepsilon (U_\varepsilon U_\varepsilon^*)\big(\mathbbm{1}+u_\varepsilon (-\Delta+\lambda\mathbbm{1})^{-1}v_\varepsilon\big)^{-1}u_\varepsilon(-\Delta+\lambda\mathbbm{1})^{-1} \\
    &=\;(\varepsilon^{3-\sigma}\eta(\varepsilon))^{-\frac{1}{2}}(-\Delta+\lambda\mathbbm{1})^{-1} v_\varepsilon \,U_\varepsilon \;\times \\
    &\qquad \times\; \varepsilon^{3-\sigma}\eta(\varepsilon)\big(\mathbbm{1}+\varepsilon^{2-\sigma}\eta(\varepsilon) u(-\Delta+\lambda\mathbbm{1})^{-1}v\big)^{-1} \times \\
    & \qquad \times\;U_\varepsilon^*u_\varepsilon(-\Delta+\lambda\mathbbm{1})^{-1}(\varepsilon^{3-\sigma}\eta(\varepsilon))^{-\frac{1}{2}}
   \end{split}
  \]
  (whenever the division by non-zero $\eta(\varepsilon)$ is possible).
  Owing to the definition \eqref{eq:defFeps}, 
  \[
    \varepsilon^{3-\sigma}\eta(\varepsilon)\big(\mathbbm{1}+\varepsilon^{2-\sigma}\eta(\varepsilon) u(-\Delta+\lambda\mathbbm{1})^{-1}v\big)^{-1}\;=\;F_\varepsilon^{(\lambda)}\,.
  \]
  By compactness and because of the choice of $\lambda$, $\mathbbm{1}+\varepsilon^{2-\sigma}\eta(\varepsilon) u(-\Delta+\lambda\mathbbm{1})^{-1}v$ is indeed invertible and its inverse is everywhere defined and bounded: so too is therefore $F_\varepsilon^{(\lambda)}$.  
 Beside, setting
 \[
  \begin{split}
     A_\varepsilon^{(\lambda)}\;&:=\;(\varepsilon^{3-\sigma}\eta(\varepsilon))^{-\frac{1}{2}}(-\Delta+\lambda\mathbbm{1})^{-1} v_\varepsilon \,U_\varepsilon\,, \\
     C_\varepsilon^{(\lambda)}\;&:=\;U_\varepsilon^*u_\varepsilon(-\Delta+\lambda\mathbbm{1})^{-1}(\varepsilon^{3-\sigma}\eta(\varepsilon))^{-\frac{1}{2}}\,,
  \end{split}
 \]
 one has 
 \[
  (H_\varepsilon+\lambda\mathbbm{1})^{-1}-(-\Delta+\lambda\mathbbm{1})^{-1}\;=\;-A_\varepsilon^{(\lambda)} F_\varepsilon^{(\lambda)} C_\varepsilon^{(\lambda)}\,,
 \]
 that is, \eqref{eq:AFCeps}, modulo checking that $A_\varepsilon^{(\lambda)}$ and $C_\varepsilon^{(\lambda)}$ defined above are indeed Hilbert-Schmidt operators with integral kernels \eqref{eq:defAepsCeps}. In fact, for $f\in L^2(\mathbb{R}^3)$,
 \[
  \begin{split}
   \big(A_\varepsilon^{(\lambda)} f\big)(x)\;&=\;(\varepsilon^{3-\sigma}\eta(\varepsilon))^{-\frac{1}{2}}\bigg((-\Delta+\lambda\mathbbm{1})^{-1} \sqrt{\frac{\eta(\varepsilon)}{\varepsilon^\sigma}\Big|V\Big(\frac{\cdot}{\varepsilon}\Big)\Big|}\,\frac{1}{\:\varepsilon^{\frac{3}{2}}} f\Big(\frac{\cdot}{\varepsilon}\Big)\bigg)(x) \\
   &=\;\int_{\mathbb{R}^3}G_\lambda(x-\varepsilon y)\sqrt{|V(y)|}\,f(y)\,\ud y\,,
  \end{split}
 \]
 and the check for the kernel of $C_\varepsilon^{(\lambda)}$ is completely analogous. Concerning the Hilbert-Schmidt property,
 \[
  \big\|A_\varepsilon^{(\lambda)} \big\|_{\mathrm{H.S.}}^2\;=\;\iint_{\mathbb{R}^3\times\mathbb{R}^3}\ud x\,\ud y\,\frac{\,|V(y)|\,e^{-2|x-\varepsilon y|\sqrt{\lambda}}\,}{(4\pi)^2|x-\varepsilon y|^2}\;=\;\frac{1}{\,8\pi\sqrt{\lambda}}\,\|V\|_{L^1}\,,
 \]
 and the very same bound can be analogously proved for $C_\varepsilon^{(\lambda)}$.
  \end{proof}

  \begin{proposition}\label{prop:AClimits}
   Let $V\in L^1(\mathbb{R}^3)$, $\lambda>0$, $\varepsilon\in(0,1]$. Then, for the operators $A_\varepsilon^{(\lambda)}$ and $C_\varepsilon^{(\lambda)}$ defined in \eqref{eq:defAepsCeps} one has
   \begin{equation}
    \begin{split}
     A_\varepsilon^{(\lambda)}\;&\xrightarrow{\;\varepsilon\downarrow 0\;}\;A^{(\lambda)}:=\;\big|G_\lambda\big\rangle\big\langle \sqrt{|V|\,}\big|\,, \\
     C_\varepsilon^{(\lambda)}\;&\xrightarrow{\;\varepsilon\downarrow 0\;}\;C^{(\lambda)}:=\;\big|\mathrm{sign}(V)\sqrt{|V|\,}\big\rangle\big\langle G_\lambda\big|
    \end{split}
   \end{equation}
   in the Hilbert-Schmidt norm.   
  \end{proposition}
  
  \begin{proof}
   By construction $A^{(\lambda)}$ has unit rank, and
   \[
     \big\|A^{(\lambda)} \big\|_{\mathrm{H.S.}}^2\;=\;\iint_{\mathbb{R}^3\times\mathbb{R}^3}\ud x\,\ud y\,\frac{\,|V(y)|\,e^{-2|x|\sqrt{\lambda}}\,}{(4\pi)^2|x|^2}\;=\;\frac{1}{\,8\pi\sqrt{\lambda}}\,\|V\|_{L^1}\;=\;\big\|A_\varepsilon^{(\lambda)} \big\|_{\mathrm{H.S.}}^2\,,
   \]
  meaning that in the limit $\varepsilon\downarrow 0$ the Hilbert-Schmidt norm is preserved. Moreover, by dominated convergence (at the level of the integral kernels), 
  \[
   A_\varepsilon^{(\lambda)}\;\xrightarrow{\;\varepsilon\downarrow 0\;}\;A^{(\lambda)}\qquad\textrm{weakly in the Hilbert space operator sense}\,.
  \]
  The two conditions above then imply that the limit actually holds in the Hilbert-Schmidt norm (see, e.g., \cite[Theorem 2.21]{simon_trace_ideals}). The very same reasoning applies to $C_\varepsilon^{(\lambda)}$.  
  \end{proof}

   \begin{remark}
   When $V\geqslant 0$, $C_\varepsilon^{(\lambda)}=\big(A_\varepsilon^{(\lambda)}\big)^*$ and $C^{(\lambda)}=\big(A^{(\lambda)}\big)^*$.
  \end{remark}

  \begin{proposition}\label{prop:Fepslimit}
  Consider the operators $F_\varepsilon^{(\lambda)}$ defined in \eqref{eq:defFeps} for given $\varepsilon\in(0,1]$,
  given $V\in\mathcal{R}$, given smooth function $\eta:[0,1]\to[0,+\infty)$ such that $\eta(0)=\eta(1)=1$, and given $\lambda>0$ such that $-\lambda\notin\sigma(H_\varepsilon)$ (where $H_\varepsilon=-\Delta+V_\varepsilon$ as a form sum with form domain $H^1(\mathbb{R}^3)$, $V_\varepsilon$ being given by \eqref{eq:general_scaling} for $\sigma>0$).
  \begin{enumerate}
   \item[(i)] Let $\sigma\in[0,2]$. When $\sigma=2$, assume in addition that $-1$ is \emph{not} an eigenvalue of $u(-\Delta)^{-1}v$.
   Then 
   \begin{equation}
    F_\varepsilon^{(\lambda)}\;\xrightarrow{\;\varepsilon\downarrow 0\;}\;\mathbb{O}
   \end{equation}
   in the norm operator sense.
   \item[(ii)] Let $\sigma=2$ and moreover assume that $(1+|\cdot|)V\in L^1(\mathbb{R}^3)$ and that $V$ satisfies the resonance condition \eqref{eq:ZERcondphi}-\eqref{eq:ZERcondphi2} for some $\phi\in L^2(\mathbb{R}^3)$ with the normalisation \eqref{eq:phinormalisation}.
   Then
   \begin{equation}
    F_\varepsilon^{(\lambda)}\;\xrightarrow{\;\varepsilon\downarrow 0\;}\;\bigg(\eta'(0)-\frac{\sqrt{\lambda}\,}{4\pi}\Big(\int_{\mathbb{R}^3}v\phi\,\ud x\Big)^2\bigg)^{-1}\,\big|\phi\big\rangle\big\langle \mathrm{sign}(V)\phi\big|
   \end{equation}
  in operator norm.
   \item[(iii)] Let $\sigma\in(2,3]$ and assume $V>0$ almost everywhere. 
   %When $\sigma=3$ assume also $V\in$ \textcolor{red}{XXX}. 
   Then 
   \begin{equation}\label{eq:Feps-strongly-vanishing}
    F_\varepsilon^{(\lambda)}\;\xrightarrow{\;\varepsilon\downarrow 0\;}\;\mathbb{O}
   \end{equation}
   strongly in the Hilbert space operator sense.
     \end{enumerate}
  \end{proposition}

   \begin{proof}
    Part (i) for $\sigma=2$ in the non-resonant case and part (ii) are classical results established first in \cite{AHK-1981-JOPTH} -- we refer to \cite[Lemma I.1.2.4]{albeverio-solvable}. 
    The proof will then focus on the regimes $\sigma\in[0,2)\cup(2,3]$.

    Observe that a standard application of dominated convergence implies
    \[
     %B_{\varepsilon^2\lambda}\;:=\;
     v(-\Delta+\varepsilon^2\lambda\mathbbm{1})^{-1}v\;\xrightarrow{\;\varepsilon\downarrow 0\;}\;v(-\Delta)^{-1}v\
     %;=\;B_0
    \]
    in the Hilbert-Schmidt norm, owing to the assumption $V\in\mathcal{R}$. When $\sigma\in[0,2)$ one then has 
    \[
     \mathbbm{1}+\varepsilon^{2-\sigma}\eta(\varepsilon)u(-\Delta+\varepsilon^2\lambda\mathbbm{1})^{-1}v\;\xrightarrow{\;\varepsilon\downarrow 0\;}\;\mathbbm{1}
    \]
    in the Hilbert-Schmidt (and hence also operator) norm, whence
    \[
     F_\varepsilon^{(\lambda)}\;=\;\varepsilon^{3-\sigma}\eta(\varepsilon)\Big( \mathbbm{1}+\varepsilon^{2-\sigma}\eta(\varepsilon)u(-\Delta+\varepsilon^2\lambda\mathbbm{1})^{-1}v\Big)^{-1}\;\xrightarrow{\;\varepsilon\downarrow 0\;}\;\mathbb{O}
    \]
    in the operator norm.

    In the regime $\sigma\in(2,3]$, owing to the further assumption of positivity of $V$,
    \[
     F_\varepsilon^{(\lambda)}\;=\;\varepsilon\,\eta(\varepsilon)\Big(\varepsilon^{\sigma-2}\mathbbm{1}+\eta(\varepsilon)v(-\Delta+\varepsilon^2\lambda\mathbbm{1})^{-1}v\Big)^{-1}\,.
    \]
    It is convenient to compare $F_\varepsilon^{(\lambda)}$ to
    \[
     \mathcal{F}_\varepsilon\;:=\;\varepsilon\,\eta(\varepsilon)\Big(\varepsilon^{\sigma-2}\mathbbm{1}+\eta(\varepsilon)v\big(2(-\Delta)+\mathbbm{1}\big)^{-1}v\Big)^{-1}\,.
    \]
    In fact, when $\varepsilon^2\lambda<1$,
    \[
     v(-\Delta+\varepsilon^2\lambda\mathbbm{1})^{-1}v\;\geqslant\; v\big(2(-\Delta)+\mathbbm{1}\big)^{-1}v\;\geqslant\;\mathbb{O}\,,
    \]
  whence
  \[
   \mathbb{O}\;\leqslant\;F_\varepsilon^{(\lambda)}\;\leqslant\;\mathcal{F}_\varepsilon\,.
  \]
   This implies, for generic $f\in L^2(\mathbb{R}^3)$,
   \[
    \begin{split}
     \big\| F_\varepsilon^{(\lambda)}f\big\|_{L^2}\;&=\sup_{ \substack{ g\in L^2(\mathbb{R}^3) \\ \|g\|_{L^2}=1} }\big|\big\langle g,F_\varepsilon^{(\lambda)}f\big\rangle_{L^2}\big| \;\leqslant\sup_{ \substack{ g\in L^2(\mathbb{R}^3) \\ \|g\|_{L^2}=1} }\Big\| \big(F_\varepsilon^{(\lambda)}\big)^{\frac{1}{2}}g\Big\|_{L^2} \;\Big\| \big(F_\varepsilon^{(\lambda)}\big)^{\frac{1}{2}}f\Big\|_{L^2} \\
     &\leqslant\sup_{ \substack{ g\in L^2(\mathbb{R}^3) \\ \|g\|_{L^2}=1} }\Big\| \mathcal{F}_\varepsilon^{\frac{1}{2}} g\Big\|_{L^2} \;\Big\| \mathcal{F}_\varepsilon^{\frac{1}{2}}f\Big\|_{L^2} \;=\;\big\|\mathcal{F}_\varepsilon^{\frac{1}{2}}\big\|_{\mathrm{op}}\big\| \mathcal{F}_\varepsilon^{\frac{1}{2}}f\big\|_{L^2}\,,
    \end{split}
   \]
   having used the inequality $\mathbb{O}\leqslant F_\varepsilon^{(\lambda)}\leqslant\mathcal{F}_\varepsilon$ in the third step.

   Concerning the norm $\big\| \mathcal{F}_\varepsilon^{\frac{1}{2}}f\big\|_{L^2}$\,,
   \[
    \begin{split}
    \big\| \mathcal{F}_\varepsilon^{\frac{1}{2}}f\big\|^2_{L^2}\;&=\;\int_{[0,+\infty)}\Big| \frac{\varepsilon\,\eta(\varepsilon)}{\,\varepsilon^{\sigma-2}+\eta(\varepsilon)t}\Big|\,\ud\mu_f^{(B)}(t) \\
      &=\;\varepsilon^{3-\sigma}\mu_f^{(B)}(\{0\})+\int_{(0,+\infty)}\Big| \frac{\varepsilon\,\eta(\varepsilon)}{\,\varepsilon^{\sigma-2}+\eta(\varepsilon)t}\Big|\,\ud\mu_f^{(B)}(t)\,,     
    \end{split}
   \]
  having set $B:=v\big(2(-\Delta)+\mathbbm{1}\big)^{-1}v$, a \emph{positive} self-adjoint Hilbert-Schmidt operator on $L^2(\mathbb{R}^3)$, and having denoted with $\mu_f^{(B)}$ the scalar spectral measure associated with $B$ and $f$. Actually,
   \[
    \mu_f^{(B)}(\{0\})\;=\;\big\langle f, E^{(B)}(\{0\}))f\big\rangle_{L^2}\;=\;0
   \]
  (where $E^{(B)}(\cdot))$ is the projection-valued measure associated to $B$) as $E^{(B)}(\{0\}))=\mathbb{O}$, because the spectral value zero is \emph{not} an eigenvalue for $B$. The latter fact is a consequence of the positivity of $V$ \emph{almost everywhere}, which allows to deduce from $0=Bg=v\big(2(-\Delta)+\mathbbm{1}\big)^{-1}v g$ precisely $g=0$. Therefore,
  \[
  \big\| \mathcal{F}_\varepsilon^{\frac{1}{2}}f\big\|^2_{L^2}\;=\;\int_{(0,+\infty)}\Big| \frac{\varepsilon\,\eta(\varepsilon)}{\,\varepsilon^{\sigma-2}+\eta(\varepsilon)t}\Big|\,\ud\mu_f^{(B)}(t)
  \]
  and since $t>0$ dominated convergence is applicable (with integrable majorant given by the constant function $\|\eta\|_{L^\infty}\mathbf{1}$ inside the integral), and
  \[
 \big\| \mathcal{F}_\varepsilon^{\frac{1}{2}}f\big\|_{L^2}\;\;\xrightarrow{\;\varepsilon\downarrow 0\;}\;0\qquad \forall f\in L^2(\mathbb{R}^3)\,.
  \]
  This shows that $\mathcal{F}_\varepsilon^{\frac{1}{2}}\xrightarrow{\,\varepsilon\downarrow 0\,}\mathbb{O}$ in the strong operator topology, whence also 
  \[
   \big\|\mathcal{F}_\varepsilon^{\frac{1}{2}}\big\|_{\mathrm{op}}\;\leqslant\; C\;<\;+\infty\qquad\textrm{uniformly in $\varepsilon$}\,,
  \]
  owing to the uniform boundedness principle.

  The latter findings yield
  \[
   \big\| F_\varepsilon^{(\lambda)}f\big\|_{L^2}\;\leqslant\;\big\|\mathcal{F}_\varepsilon^{\frac{1}{2}}\big\|_{\mathrm{op}}\big\| \mathcal{F}_\varepsilon^{\frac{1}{2}}f\big\|_{L^2}\;\;\xrightarrow{\;\varepsilon\downarrow 0\;}\;0\qquad \forall f\in L^2(\mathbb{R}^3)\,,
  \]
  which is precisely the limit \eqref{eq:Feps-strongly-vanishing}.
   \end{proof}

  \begin{proof}[Proof of Theorem \ref{thm:res-convergence}]
   We already commented that the whole case $\sigma=2$ is already present in the literature (see., e.g., \cite[Theorem I.1.2.5]{albeverio-solvable}): we re-obtain it here for completeness of presentation and for comparison purposes with respect to the other values of the parameter $\sigma$.
   
   (i) Under the assumptions for this first part of the theorem, 
   %\emph{and introducing the additional condition that $V>0$ for $\sigma\in(2,3]$}, 
   we see that Propositions \ref{prop:AFC}, \ref{prop:AClimits}, and \ref{prop:Fepslimit}(i) and (iii) imply
   \[
    (H_\varepsilon+\lambda\mathbbm{1})^{-1}-(-\Delta+\lambda\mathbbm{1})^{-1}\;=\;-A_\varepsilon^{(\lambda)} F_\varepsilon^{(\lambda)} C_\varepsilon^{(\lambda)}\;\xrightarrow{\;\varepsilon\downarrow 0 \;} \;\mathbb{O}
   \]
   in the Hilbert space norm operator topology when $\sigma\in[0,2]$ and in the Hilbert space strong operator topology when $\sigma\in(2,3]$
   $\lambda>0$ with $-\lambda\notin\sigma(H_\varepsilon)$ for each sufficiently small $\varepsilon$
   (and any $\lambda\in\mathbb{C}\setminus\mathbb{R}$). This is the limit \eqref{eq:resolventlimit1}.

%    under the restriction of strict positivity of $V$ when $\sigma\in(2,3]$. 
%    
%    
%    Such a restriction is removed as follows. For given $\varepsilon\in(0,1]$ and given non-negative $V\in\mathcal{R}\cap L^1(\mathbb{R}^3)$, as in the assumptions of part (i) when $\sigma\in(2,3]$, set
%    \[
%     V^{(\varepsilon)}(x)\;:=\;V(x)+\varepsilon^4 e^{-|x|^2}\mathbf{1}_{\mathbb{R}\setminus\mathrm{supp}(V)}\,,
%    \]
%    with the customary notation $\mathbf{1}_\Omega$ for the characteristic function of a subset $\Omega\subset\mathbb{R}^3$. By construction $V^{(\varepsilon)}$ is strictly positive everywhere and $V^{(\varepsilon)}\in\mathcal{R}\cap L^1(\mathbb{R}^3)$. Thus, setting
%    \[
%     \widetilde{V}_\varepsilon(x)\;;=\;\frac{\eta(\varepsilon)}{\:\varepsilon^\sigma}V^{(\varepsilon)}
%    \]

   (ii) Under the assumptions for this first part of the theorem, Propositions \ref{prop:AFC}, \ref{prop:AClimits}, and \ref{prop:Fepslimit}(ii) now imply
   \[
    \begin{split}
     (H_\varepsilon+\lambda\mathbbm{1})^{-1}-(-\Delta+\lambda\mathbbm{1})^{-1}\;&=\;-A_\varepsilon^{(\lambda)} F_\varepsilon^{(\lambda)} C_\varepsilon^{(\lambda)} \\
     & \quad\; \xrightarrow{\;\varepsilon\downarrow 0 \;}\;-\frac{\big(\int_{\mathbb{R}^3}v\phi\,\ud x\big)^2}{\: \eta'(0)-\frac{\sqrt{\lambda}\,}{4\pi}\big(\int_{\mathbb{R}^3}v\phi\,\ud x\big)^2 \:}\,|G_\lambda\rangle\langle G_\lambda|
    \end{split}
   \]
   in the Hilbert space norm operator topology.
   Dividing in the r.h.s.~by $\big(\int_{\mathbb{R}^3}v\phi\,\ud x\big)^2=\big(\int_{\mathbb{R}^3}V\psi\,\ud x\big)^2\neq 0$ (see \eqref{eq:psiL2locnotL2} above) and defining $\alpha$ as in \eqref{eq:thm2-def-alpha}, the above limit reads
   \[
     (H_\varepsilon+\lambda\mathbbm{1})^{-1}\;\xrightarrow{\;\varepsilon\downarrow 0 \;}\;(-\Delta+\lambda\mathbbm{1})^{-1}+\frac{1}{\alpha+\frac{\sqrt{\lambda}}{4\pi}}\,|G_\lambda\rangle\langle G_\lambda|\,.
   \]
  The r.h.s.~is precisely the resolvent $(-\Delta_\alpha+\lambda\mathbbm{1})^{-1}$, see \eqref{eq:resolventDalpha}. The limit \eqref{eq:resolventlimit2} is thus established.  
  \end{proof}

  \begin{remark}\label{rem:sigma-above3}
   The factorised structure \eqref{eq:AFCeps} for the resolvent difference
   \[
    (H_\varepsilon+\lambda\mathbbm{1})^{-1}-(-\Delta+\lambda\mathbbm{1})^{-1}
   \]
   is always valid, irrespective of $\sigma\geqslant 0$ (Proposition \ref{prop:AFC}). What constrains $\sigma$ not to exceed the threshold value $\sigma=3$ is the control of the limit in the $F_\varepsilon^{(\lambda)}$-term as performed in Proposition \ref{prop:Fepslimit}, which is the actual crucial step for the current proof of Theorem \ref{thm:res-convergence}. Indeed, as is evident from the magnitude of the powers of $\varepsilon$, $F_\varepsilon^{(\lambda)}$ ceases to have a definite strong limit as long as $\sigma>3$, and in particular the quantity $\varepsilon^{3-\sigma}\mu_f^{(B)}(\{0\})$ singled out in the proof of Proposition \ref{prop:Fepslimit} becomes uncontrollable. This invalidates the scheme where the strong vanishing of $A_\varepsilon^{(\lambda)}F_\varepsilon^{(\lambda)}C_\varepsilon^{(\lambda)}$ is established by a \emph{separate} analysis of the norm limits of $A_\varepsilon^{(\lambda)}$ and $C_\varepsilon^{(\lambda)}$, and of the strong limit of $F_\varepsilon^{(\lambda)}$. This also naturally suggests that in super-Colombeau regimes $\sigma>3$ the operator $A_\varepsilon^{(\lambda)}F_\varepsilon^{(\lambda)}C_\varepsilon^{(\lambda)}$ is to be treated \emph{as a whole}. So, concretely speaking, one should rather investigate the limiting behaviour of $F_\varepsilon^{(\lambda)}$ on the range of $C_\varepsilon^{(\lambda)}$, and the like. Beside, a more manageable control would be expected in terms of the integral kernels, which would lead to establish \emph{weak} resolvent convergence: in turn, by standard resolvent formulas, this would imply again strong resolvent convergence (see, e.g., \cite[Problem VIII.20]{rs1}). We defer the presentation of our reasonings along this conceptual line to a future publication.   
  \end{remark}

  \section{Proof of Theorems \ref{thm:main1} and \ref{thm:Colombeau-case}}\label{sec:proofmainresults}

  Let us discuss first the linear case.

  \begin{proof}[Proof of Theorem \ref{thm:main1} -- linear case $w\equiv \mathbf{0}$]
   Owing to the resolvent convergence obtained in Theorem \ref{thm:res-convergence}, which for the present purposes is sufficient to have in the \emph{strong} operator sense, and applying Trotter's theorem (see, e.g., \cite[Theorem VIII.21 and Problem VIII.21]{rs1}), we can equivalently re-write the limits \eqref{eq:resolventlimit1} and \eqref{eq:resolventlimit2}, respectively, as
     \[
    \big\| e^{-\ii t H_\varepsilon} a - e^{\ii t \Delta} a \big\|_{L^\infty([0,\textsf{T}],L^2_x)}\;\xrightarrow{\;\varepsilon\downarrow 0\;}\; 0\qquad \textrm{(case (i))}
   \]
  and 
   \[
    \big\| e^{-\ii t H_\varepsilon} a - e^{\ii t \Delta_\alpha} a \big\|_{L^\infty([0,\textsf{T}],L^2_x)}\;\xrightarrow{\;\varepsilon\downarrow 0\;}\; 0\qquad \!\textrm{(case (ii))}
   \]
   for any initial datum $a\in L^2(\mathbb{R}^3)$. The linear case is thus proved.   
  \end{proof}

  \begin{remark}
   Whereas we used the known property (Trotter's theorem) for which strong resolvent convergence between self-adjoint operators is equivalent to the strong convergence of the unitary groups (propagators) generated by them,
   the convergence of such propagators \emph{in operator norm} cannot hold (as emerges, e.g., from the proof of \cite[Theorem VIII.20 and Problem VIII.29]{rs1}).
  \end{remark}

  Next, let us discuss the non-linear case.

  \begin{proof}[Proof of Theorem \ref{thm:main1} -- non-linear case $w \equiv\!\!\!\!\!/ \;\:\mathbf{0}$]
   Let us consider non-restrictively the forward in time evolution ($t>0$) and re-write the differential problems under considerations, namely
   \[
    \begin{cases}
    \;\;\; \ii\partial_t u_\varepsilon \,=\, H_\varepsilon u_\varepsilon + (w*|u_\varepsilon|^2) u_\varepsilon\,, \\
     u_\varepsilon(0,\cdot)\,=\,a \quad(\forall\varepsilon>0)     
    \end{cases}
   \]
   and 
   \[
    \begin{cases}
    \;\;\; \ii\partial_t u \,=\, H^{\circ} u + (w*|u|^2) u\,, \\
     u(0,\cdot)\,=\,a \,,
    \end{cases}
   \]
   in integral form, that is, respectively,
   \[
    u_\varepsilon(t)\;=\;e^{-\ii t H_\varepsilon} a - \ii \int_0^t \big(e^{-\ii (t-s) H_\varepsilon} (w*|u_\varepsilon|^2)u_\varepsilon\big)(s)\,\ud s
   \]
   and 
   \[
    u(t)\;=\;e^{-\ii t H^{\circ}} a - \ii\int_0^t \big(e^{-\ii (t-s) H^{\circ}} (w*|u|^2)u\big)(s)\,\ud s\,.
   \]
  To unify the treatment, in the above expressions $H^{\circ}$ stays for either the self-adjoint negative Laplacian $-\Delta$ or for the self-adjoint point interaction Hamiltonian $-\Delta_\alpha$, depending on case (i) or (ii) of the theorem. Beside, we make the customary abuse of notation writing $u(t)$ for $u(t,\cdot)$, and the like. The above expressions are all meaningful, owing to the well-posedness provided by Theorem \ref{thm:well-posedness}.
   
   We have
    \begin{equation*}\tag{a}\label{eq:diff-split}
  u_\varepsilon(t)-u(t)\;=\;P_\varepsilon(t)+Q_\varepsilon(t)+R_\varepsilon(t)\,,
 \end{equation*}
 where
 \begin{equation*}
  \begin{split}
   P_\varepsilon(t)\;&:=\;e^{-\ii t H_\varepsilon}a-e^{-\ii t H^{\circ}}a\,, \\
   Q_\varepsilon(t)\;&:=\;-\ii \int_0^t e^{-\ii (t-s) H_\varepsilon}\big( (w*|u_\varepsilon|^2)u_\varepsilon-(w*|u|^2)u\big)(s)\,\ud s\,, \\
   R_\varepsilon(t)\;&:=\;-\ii\int_0^t \Big(e^{-\ii (t-s) H_\varepsilon}-e^{-\ii (t-s) H^{\circ}}\Big)(w*|u|^2)u(s)\,\ud s\,.
  \end{split}
 \end{equation*}

   Concerning $P_\varepsilon$,
   \[\tag{b}\label{eq:Peps}
    \lim_{\varepsilon\downarrow 0}\big\| P_\varepsilon\big\|_{L^\infty([0,\textsf{T}],L^2_x)}=\,0\,,
   \]
  as obtained already in the proof for the linear problem.

  Concerning $R_\varepsilon$, the H\"{o}lder and Young inequalities and the preservation of the $L^2$-norm of the solution (Theorem \ref{thm:well-posedness}) yield (at any $t$)
  \[
   \begin{split}
       \big\| (w*|u |^2)u  \big\|_{L^2_x}\;&\leqslant\;\big\| w*|u |^2 \big\|_{L^\infty_x}\|u \|_{L^2_x} \\
       &\lesssim\;\|w\|_{L^\infty_x}\|u \|^3_{L^2_x} \\
       &=\;\|w\|_{L^\infty_x}\|a\|^3_{L^2_x}\,,
   \end{split}
  \]
  meaning that $(w*|u|^2)u\in L^2(\mathbb{R}^3)$ at any time, and therefore, again on account of the proof for the linear problem,
  \[\tag{c}\label{eq:Reps}
    \lim_{\varepsilon\downarrow 0}\big\| R_\varepsilon\big\|_{L^\infty([0,\textsf{T}],L^2_x)}=\,0\,.
   \]

 Last, concerning $Q_\varepsilon$, and for $T\in(0,\textsf{T}]$ to be specified in a moment,
  \begin{equation*}
   \begin{split}
   \big\| Q_\varepsilon\big\|_{L^\infty([0,T],L^2_x)} &\leqslant\; T \,\big\| (w*|u_\varepsilon|^2)u_\varepsilon-(w*|u|^2)u\big\|_{L^\infty([0,T],L^2_x)} \,,
   \end{split}
  \end{equation*}
  owing to $L^2$-unitarity. In the above r.h.s.,
  \[
   \begin{split}
   \big|(w*|u_\varepsilon|^2)u_\varepsilon-(w*|u|^2)u \big| \; &\leqslant\; (|w|*|u_\varepsilon|^2)|u_\varepsilon-u| \\
   & \quad + \big( |w|*\big(|u_\varepsilon-u|(|u_\varepsilon|+|u|)\big)\big)|u|\,.
   \end{split}
  \]
   The Young inequality and the conservation of the $L^2$-norm of the solutions allow to deduce
   \begin{equation*}\tag{d}\label{eq:Qeps}
     \big\| Q_\varepsilon\big\|_{L^\infty([0,T],L^2_x)}\;\leqslant\; 3\,T\, \|w\|_{L^\infty_x}\|a\|_{L^2_x}^2,\|u_\varepsilon-u\|_{L^\infty([0,T],L^2_x)}\,.
   \end{equation*}
   Selecting $T$ so as $T<(3\|w\|_{L^\infty_x}\|a\|_{L^2_x}^2)^{-1}$, \eqref{eq:diff-split} and \eqref{eq:Qeps} together yield
   \begin{equation*}
    \|u_\varepsilon-u\|_{L^\infty([0,T],L^2_x)}\;\leqslant\;\frac{\;\big\| P_\varepsilon\big\|_{L^\infty([0,T],L^2_x)}+\big\| R_\varepsilon\big\|_{L^\infty([0,T],L^2_x)}}{\big(1-3 \,T\, \|w\|_{L^\infty_x}\|a\|_{L^2_x}^2\big)}
   \end{equation*}
 and \eqref{eq:Peps}-\eqref{eq:Reps} then imply 
 \begin{equation*}
   \lim_{\varepsilon\downarrow 0}\|u_\varepsilon-u\|_{L^\infty([0,T],L^2_x)}=\;0\,.
 \end{equation*}

  We then repeat the above reasoning finitely many times by partitioning the interval $[0,\mathsf{T}]$ into intervals $[kT,(k+1)T]\cap[0,\mathsf{T}]$, $k=0,1,2,\dots$.
  For each such step, the difference 
  \[
   \|u_\varepsilon-u\|_{L^\infty([kT,(k+1)T],L^2_x)}
  \]
  is investigated with a splitting that is completely analogous to the above terms $P_\varepsilon$, $Q_\varepsilon$, $R_\varepsilon$, where clearly the initial time is now $t=kT$ instead of $t=0$ as in the first step. Obviously, an additional term arises, of the form
  \[
   e^{-\ii (t-kT)H^{\circ}} \big( u_\varepsilon(kT)-u(kT) \big)\,,
  \]
  due to the fact that the initial data for the $k$-th interval are different for $u_\varepsilon$ and $u$: its $L^2_x$-vanishing follows from the uniform control on the $(k-1)$-th time interval.

  The final conclusion is
  \begin{equation*}
   \lim_{\varepsilon\downarrow 0}\|u_\varepsilon-u\|_{L^\infty([0,\mathsf{T}],L^2_x)}=\;0\,.
 \end{equation*}
  The limits \eqref{eq:mainthmlimit1} and \eqref{eq:mainthmlimit2} are thus proved.  
  \end{proof}

  Last, for the case $\sigma=3$, let us establish compatibility in the sense of the Colombeau generalised solution theory.

 \begin{proof}[Proof of Theorem \ref{thm:Colombeau-case}]
  In the formalism of Colombeau generalised solutions (Appendix \ref{app:Colombeau}), we have to consider the net $(u_\varepsilon)_{\varepsilon\in(0,1]}$, 
  %(requiring $\varepsilon>0$ to belong to $(0,1]$ is clearly non-restrictive), 
  where each $u_\varepsilon$ is the unique $L^2$-solution to
  \[
  \begin{cases}
   \ii\partial_t u_\varepsilon\,=\,-\Delta u_\varepsilon +\phi_\varepsilon u_\varepsilon+(w*|u_\varepsilon|^2)u_\varepsilon\,, \\
   u_\varepsilon(0,\cdot)\,=\,a_\varepsilon\;:=\;a*\rho_\varepsilon\,.
  \end{cases}
\]
  Here, for every $x\in\mathbb{R}^3$ and $\varepsilon\in(0,1]$,
   \begin{equation*}
  \rho_\varepsilon(x)\,:=\,\frac{1}{\,\varepsilon^3}\,\rho\Big(\frac{x}{\varepsilon}\Big)
 \end{equation*}
 for some $\rho\in\mathcal{S}(\mathbb{R}^3)$ such that $\int_{\R^3}\rho(x)\,\ud x=1$ (see \eqref{eq:colemb}), 
  and  
  \[
   \phi_\varepsilon(x)\;:=\;\frac{1}{\,\varepsilon^3}\,\phi\Big(\frac{x}{\varepsilon}\Big)\,,
   \]
  for an arbitrary $\phi\in \mathcal{S}(\mathbb{R}^3)$ with $\phi> 0$ almost everywhere and $\int_{\mathbb{R}^3}\phi(x)\,\ud x=1$.

%   The existence of such a $u_\varepsilon$ is guaranteed by Theorem \ref{thm:well-posedness}. Thus, $(u_\varepsilon)_{\varepsilon\in(0,1]}$ is the (representative of a) Colombeau solution to \eqref{eq:deltaeq} with initial datum $a$, and its \emph{uniqueness as a Colombeau solution}, i.e., up to suitable quotienting (Appendix \ref{app:Colombeau}) was already established in \cite{Dugandzija-Vojnovic-2021}. 

  Compatibility with \eqref{eq:nodeltaeq} would mean
  \[\tag{*}\label{eq:proof-comp}
     \lim_{\varepsilon\downarrow 0}\|u_\varepsilon-u\|_{L^\infty([0,\mathsf{T}],L^2_x)}=\;0
     %\lim_{\varepsilon\downarrow 0}\|u_\varepsilon-u\|_{L^\infty(\mathbb{R}_t,L^2_x)}\;=\;0\,,
 \]
   for arbitrary $\mathsf{T}>0$, where $u$ is the unique solution in $C(\mathbb{R},L^2(\mathbb{R}^3))$ to the Cauchy problem
     \[
  \begin{cases}
   \ii\partial_t u\,=\,-\Delta  u+(w*|u|^2)u\,, \\
   u(0,\cdot)\,=\,a
  \end{cases}
\]
    (Theorem \ref{thm:well-posedness} again guaranteeing the existence of such $u$).

    In order to establish compatibility, 
    observe that 
%     
%     let us consider for each $\varepsilon>0$ the unique $L^2$-solution $\widetilde{u}_\varepsilon$ (provided by Theorem \ref{thm:well-posedness}) to the auxiliary problem
%       \[
%   \begin{cases}
%    \ii\partial_t \widetilde{u}_\varepsilon\,=\,-\Delta \widetilde{u}_\varepsilon +\phi_\varepsilon \widetilde{u}_\varepsilon+(w*|\widetilde{u}_\varepsilon|^2)\widetilde{u}_\varepsilon\,, \\
%    \widetilde{u}_\varepsilon(0,\cdot)\,=\,a\,.
%   \end{cases}
% \]
%   The
%   
  the present $\phi_\varepsilon$ is a special case of the general $V_\varepsilon$ considered in Theorem \ref{thm:res-convergence}, which is therefore applicable and yields
    \[
    \big\| e^{-\ii t H_\varepsilon} f - e^{\ii t \Delta} f \big\|_{L^\infty([0,\textsf{T}],L^2_x)}\;\xrightarrow{\;\varepsilon\downarrow 0\;}\; 0\qquad \forall f\in L^2(\mathbb{R}^3)\,,
   \]
  where here $H_\varepsilon$ is the self-adjoint form sum $-\Delta+\phi_\varepsilon$ on $L^2(\mathbb{R}^3)$. We can therefore repeat the reasoning of the proof of Theorem \ref{thm:main1} for the comparison, now, between
     \[
    u_\varepsilon(t)\;=\;e^{-\ii t H_\varepsilon} a_\varepsilon - \ii \int_0^t \big(e^{-\ii (t-s) H_\varepsilon} (w*|u_\varepsilon|^2)u_\varepsilon\big)(s)\,\ud s
   \]
   and 
   \[
    u(t)\;=\;e^{\ii t \Delta} a - \ii\int_0^t \big(e^{\ii (t-s) \Delta} (w*|u|^2)u\big)(s)\,\ud s\,.
   \]
   This proceeds along the very same scheme therein (first on a suitably small interval $[0,T]$ and then covering the whole $[0,\mathsf{T}]$, except for the presence of the additional term
   \[
     e^{-\ii t H_\varepsilon}a_\varepsilon - e^{\ii t \Delta} a\,.
   \]
   Now, by construction $a_\varepsilon=\rho_\varepsilon*a\xrightarrow{\,\varepsilon\downarrow 0\,}a$ in $H^s(\mathbb{R}^3)$ for any $s\geqslant 0$, hence in particular in $L^2(\mathbb{R}^3)$. This implies
   \[
    \begin{split}
     &\big\|e^{-\ii t H_\varepsilon}a_\varepsilon - e^{\ii t \Delta} a\big\|_{L^\infty([0,\mathsf{T}],L^2_x)} \\
     &\qquad\qquad\leqslant\;\big\|e^{-\ii t H_\varepsilon}a_\varepsilon - e^{-\ii t H_\varepsilon}a\big\|_{L^\infty([0,\mathsf{T}],L^2_x)}+\big\|e^{-\ii t H_\varepsilon}a - e^{\ii t \Delta} a\big\|_{L^\infty([0,\mathsf{T}],L^2_x)} \\
     &\qquad\qquad=\;\|a_\varepsilon-a\|_{L^2_x}+\big\|e^{-\ii t H_\varepsilon}a - e^{\ii t \Delta} a\big\|_{L^\infty([0,\mathsf{T}],L^2_x)}\;\xrightarrow{\;\varepsilon\downarrow 0\;}\;0
    \end{split}
   \]
   and finally establishes \eqref{eq:proof-comp}.
 \end{proof}

 \appendix
 
 \section{The Colombeau algebra and the notions of generalised solutions and compatibility for Schr\"{o}dinger equations.}\label{app:Colombeau}

% \textcolor{red}{Here standard definitions and recap of properties.}
 
%  \textcolor{red}{
% Also:
%  \begin{itemize}
%   \item  in Rosinger \cite{Rosinger-LNM1978,Rosinger-1980,Rosinger-1987} no $\delta^2$ !!
%   \item  and, from Todorov's review of \cite{Grosser-Farkas-Kunzinger-Steinbauer-2001}, implement the part ``Following Colombeau's breakthrough....'' etc.
%  \end{itemize}
%  }
  
   % Brezis 83

 We recall in this Appendix the main features of the Colombeau algebra of generalised functions, we referred to in Section \ref{sec:intro-main} and in Theorem \ref{thm:Colombeau-case}.

  The Colombeau algebra is an associative differential algebra, and as already mentioned it displays properties that overcome the problem of multiplication of distributions, thus making it well suited as a solution framework for non-linear partial differential equations. It was introduced by Colombeau \cite{Colombeau-1984,Colombeau-1985} in the early 1980's, and it is discussed in the literature through an ample spectrum of precursors, refinements, and more recent reviews, as in \cite{Rosinger-LNM1978,Rosinger-1980,Rosinger-1987,Biagioni-LNM1990,Oberguggenberger-1992,Grosser-Farkas-Kunzinger-Steinbauer-2001,Grosser-Kunzinger-Oberguggenberger-Steinbauer-2001,Gsponer-2009} (observe, for instance, that in the precursor represented by Rosinger's notion of generalised functions \cite{Rosinger-LNM1978,Rosinger-1980}, the ``product of deltas'' $\delta^2$ was still ill-defined, whereas it is allowed within the Colombeau algebra).

  For concreteness, and for its relevance in the discussion of this work, we present here in particular the $H^2$-based Colombeau algebra, following the approach of \cite{Nedeljkov-Oberguggenberger-Pilipovic-2005,Nedeljkov-Pilipovic-RajterCiric-2005-heatdelta}. Additional details in this respect may be found in \cite{DV-2022-INdAMSpringer}. A completely analogous construction can be made for $E$-based Colombeau algebras for functional spaces $E$ that are locally convex spaces and multiplicative algebras.

  Given $T>0$, one denotes by
  \[
   {\mathcal{E}_{C^1,H^2}([0,T)\times\mathbb{R}^3)}\,,\qquad\textrm{respectively,}\qquad
{\mathcal{N}_{C^1,H^2}([0,T)\times\mathbb{R}^3)}
  \]
 the vector 
space of nets $(u_\varepsilon)_{\varepsilon\in(0,1]}$ (with natural component-wise linear structure), of functions
\[u_\varepsilon\;\in\; C([0,T), H^2(\mathbb{R}^3))\cap 
C^1([0,T), L^2(\mathbb{R}^3))
 \]
 satisfying the following growth properties as $\varepsilon\downarrow 0$: for 
 %each $u_\varepsilon$ of 
 a $\mathcal{E}_{C^1,H^2}$-net there is $N\in\mathbb{N}$ such that
 \begin{equation}
  \max\bigg\{\sup_{t\in[0,T)}\|u_\varepsilon(t)\|_{H^2}\,,\sup_{t\in[0,T)}\|
\partial_t u_\varepsilon(t)\|_2\bigg\}\;\stackrel{\varepsilon\downarrow 0}{=} \;O(\varepsilon^{-N})\,,
 \end{equation}
 and for 
 %each $u_\varepsilon$ of 
 a $\mathcal{N}_{C^1,H^2}$-net the following holds true $\forall M\in\mathbb{N}$:
 \begin{equation}
  \max\bigg\{\sup_{t\in[0,T)}\|u_\varepsilon(t)\|_{H^2}\,,\sup_{t\in[0,T)}\|
\partial_t u_\varepsilon(t)\|_2\bigg\}\;\stackrel{\varepsilon\downarrow 0}{=} \;O(\varepsilon^{M})\,.
 \end{equation}
 Elements of $\mathcal{E}_{C^1,H^2}([0,T)\times\mathbb{R}^3)$ and of $\mathcal{N}_{C^1,H^2}([0,T)\times\mathbb{R}^3)$ are called, respectively, \emph{moderate nets} and \emph{negligible nets}. 
 $\mathcal{N}_{C^1,H^2}([0,T)
\times\mathbb{R}^3)$ is a (multiplicative) bilateral ideal in $\mathcal{E}_{C^1,H^2}([0,T)
\times\mathbb{R}^3)$, and this allows to introduce the quotient space
 \begin{equation}\label{eq:Gquotient}
  \mathcal{G}_{C^1,H^2}([0,T)\times\mathbb{R}^3)\;:=\;\mathcal{E}_{C^1,H^2}([0,T)
\times\mathbb{R}^3)\,/\,\mathcal{N}_{C^1,H^2}([0,T)\times\mathbb{R}^3)
 \end{equation}
 which is a Colombeau-type vector space. In a completely analogous manner one defines
 \begin{equation}\label{eq:EH2}
  \mathcal{E}_{H^2}(\R^3)\;:=\;\left\{(u_\varepsilon)_{\varepsilon\in(0,1]}\left| 
  \begin{array}{c}
   u_\varepsilon\in H^2(\R^3)\,,\;\textrm{and for each such net} \\
   \exists\,N\in\mathbb{N}\;\textrm{such that}\,\;\| u_\varepsilon\|_{H^2}\stackrel{\varepsilon\downarrow 0}{=} O(\varepsilon^{-N})
  \end{array}
  \!\!\right.\right\},
 \end{equation}
 \begin{equation}
  \mathcal{N}_{H^2}(\R^3)\;:=\;\left\{(u_\varepsilon)_{\varepsilon\in(0,1]}\left| 
  \begin{array}{c}
   u_\varepsilon\in H^2(\R^3)\,,\;\textrm{and for each such net} \\
   \| u_\varepsilon\|_{H^2}\stackrel{\varepsilon\downarrow 0}{=} O(\varepsilon^{M})\;\;\forall M\in\mathbb{N}
  \end{array}
  \!\!\right.\right\},
 \end{equation}
 and
 \begin{equation}
  \mathcal{G}_{H^2}(\R^3)\;:=\;\mathcal{E}_{H^2}(\R^3)\,/\,\mathcal{N}_{H^2}(\R^3)\,,
 \end{equation}
 with the natural vector space structure.

 In fact, precisely because each $H^2(\mathbb{R}^d)$ with $d\in\{1,2,3\}$ is a multiplicative algebra, so too are $\mathcal{G}_{C^1,H^2}([0,T)
\times\mathbb{R}^3)$ and $\mathcal{G}_{H^2}(\mathbb{R}^3)$ with respect to the product inherited by the quotients.

 Furthermore, $\mathcal{E}_{C^1,H^2}([0,T)
\times\mathbb{R}^3)$ and $\mathcal{E}_{H^2}(\mathbb{R}^3)$ can be suitably topologised by semi-norms making them locally convex spaces, and in turn a sharp topology can be defined in $\mathcal{G}_{C^1,H^2}([0,T)
\times\mathbb{R}^3)$ and $\mathcal{G}_{H^2}(\mathbb{R}^3)$ making them topological vector spaces -- see, e.g., \cite{Garetto-2005-topologicalColombeau,Garetto-2005-topologicalColombeau2} for details.

 The Sobolev spaces $H^s(\mathbb{R}^3)$, for $s\in\mathbb{R}$, as well as the space $\mathcal{E}'(\mathbb{R}^3)$ of distributions with compact support, are identified (one also customarily says `embedded') as subspaces of $\mathcal{G}_{H^2}(\mathbb{R}^3)$ by associating to each $F\in H^s(\mathbb{R}^3)$, or $F\in\mathcal{E}'(\mathbb{R}^3)$, the equivalence class of the net $(F*\rho_\varepsilon)_{\varepsilon\in(0,1]}$ up to $\mathcal{N}_{H^2}$-nets, where
 \begin{equation}\label{eq:colemb}
  \begin{split}
     \rho_\varepsilon(x)\,&:=\,\frac{1}{\,\varepsilon^3}\,\rho\Big(\frac{x}{\varepsilon}\Big)\quad\textrm{for some }\rho\in\mathcal{S}(\mathbb{R}^3)\textrm{ such that } 
     %\rho>0\,,
     \int_{\R^3}\rho(x)\,\ud x\,=\,1\,.\\
     %&\textrm{and}\;\; \int_{\R^3} x^{\alpha} \rho(x)\,\ud x\,=\,0\quad\textrm{for every multi-index $\alpha$ with $|\alpha|\geqslant 1$\,.}
  \end{split}
 \end{equation}
 Indeed, all such $F*\rho_\varepsilon$'s are well defined, belong to $H^2(\mathbb{R}^3)$, and satisfy the growth condition \eqref{eq:EH2}. (Recall, in particular, that for $F\in\mathcal{E}'(\mathbb{R}^3)$ the object $F*\rho_\varepsilon$ is the distribution acting on test functions $\varphi$ as $(F*\rho_\varepsilon)(\varphi)=F(\rho_\varepsilon*\varphi)$.)

Next, let us review the notion of \emph{generalised Colombeau solution} to the Cauchy problem for semi-linear Schr\"{o}dinger equations with possible distributional coefficients, explicitly the problem
\begin{equation}\label{linear_eq}
 \begin{cases}
  \;\ii\partial_t u\;=\;-\Delta_x u + c\,\delta\,u+g(u)\,, \\
  u(0,\cdot)\;=\;a\in H^2(\mathbb{R}^3)
 \end{cases}
\end{equation}
 for given $c\in\mathbb{R}$, given non-linearity $g(\cdot)$ of the usual form $g(u)=(w*|u|^2)u$ or $g(u)=\gamma\,|u|^{p-1}u$ with $\gamma\in\mathbb{R}$, $p\geqslant 1$, and given measurable, real-valued $w$. Observe that \eqref{linear_eq} encompasses both the classical problem, when $c=0$, and the one with distributional (Dirac delta) coefficient in the linear part, when $c\neq 0$.

 One says that \eqref{linear_eq} admits a (local in time) solution $u\in \mathcal{G}_{C^1,H^2}([0,T)\times\R^3)$, for some $T>0$, if there are three nets 
 \begin{itemize}
  \item $(u_\varepsilon)_{\varepsilon\in(0,1]}\in\mathcal{E}_{C^1,H^2}([0,T)\times\R^3)$,
  \item $(n_\varepsilon)_{\varepsilon\in(0,1]}\in\mathcal{N}_{H^2}(\R^3)$\,,
  \item and $(M_\varepsilon)_{\varepsilon\in(0,1]}\in C([0,T),L^2(\mathbb{R}^3))$ with  $\|M_\varepsilon\|_{L^\infty([0,T),L^2(\mathbb{R}^3))}\stackrel{\varepsilon\downarrow 0}{=}O(\varepsilon^M)$  $\forall M\in\mathbb{N}$
 \end{itemize}
  satisfying
 \begin{equation}
 \begin{cases}
  \;\ii\partial_t u_\varepsilon-\big(-\Delta_x u_\varepsilon +c\,\delta_\varepsilon\,u_\varepsilon+g(u_\varepsilon)\big)\,=\,M_\varepsilon\,, \\
  u(0,\cdot)\;=\;a_\varepsilon+n_\varepsilon\,,
 \end{cases}
\end{equation}
 where
 \begin{equation}
  a_\varepsilon\,:=\,a*\rho_\varepsilon\,,\qquad \delta_\varepsilon\,:=\,\delta*\rho_\varepsilon\,,
 \end{equation}
 and $\rho_\varepsilon$ is given by \eqref{eq:colemb} (i.e., the equivalence classes $[(a_\varepsilon)_{\varepsilon\in(0,1]}]$ and $[(\delta_\varepsilon)_{\varepsilon\in(0,1]}]$ are the identifications (`embeddings') in $\mathcal{G}_{H^2}(\mathbb{R}^3)$, respectively, of $a\in H^2(\mathbb{R}^3)$ and of $\delta\in\mathcal{E}'(\mathbb{R}^3)$). This definition characterises the generalised solution $u$ independently of its representative $(u_\varepsilon)_{\varepsilon\in(0,1]}\in\mathcal{E}_{C^1,H^2}([0,T)\times\R^3)$ with respect to the quotient \eqref{eq:Gquotient}.

% 
% \begin{definition}\label{def_solution}
% We say that $u\in \mathcal{G}_{C^1,H^2}([0,T)\times\R^3)$ is a solution of $\eqref{linear_eq}$ if for an initial condition $a\in H^2(\R^3)$ and its class  represented by $(a_\varepsilon)_\varepsilon=(a*\rho_\varepsilon)_\varepsilon$, there exists a representative $(u_\varepsilon)_\varepsilon\in\mathcal{E}_{C^1,H^2}([0,T)\times\R^3)$ such that 
% \begin{align}
% \begin{split}
% i(u_\varepsilon)_t+\triangle u_\varepsilon+g(u_\varepsilon)&=M_\varepsilon,\\
% u_\varepsilon(0)&=a_\varepsilon+n_\varepsilon,\label{notion_solution}
% \end{split}
% \end{align}
% for some $n_\varepsilon\in \mathcal{N}_{H^2}(\R^3)$, where $\sup_{t\in[0,T)}\|M_\varepsilon\|_2=O(\varepsilon^M)$, $\varepsilon \to 0$, for any $M\in\mathbb{N}$.
% \end{definition}
% Definition does not depend on the representative of the class $u=[(u_\varepsilon)_\varepsilon]$.

 In addition, the problem \eqref{linear_eq} is said to admit a \emph{unique} generalised Colombeau solution if for any two $u,v\in \mathcal{G}_{C^1,H^2}([0,T)\times\R^3)$ satisfying \eqref{linear_eq} in the above generalised Colombeau sense one has
 \begin{equation}
  \|u_\varepsilon-v_\varepsilon\|_{L^\infty([0,T),L^2(\mathbb{R}^3))}\;\stackrel{\varepsilon\downarrow 0}{=}\;O(\varepsilon^M)\qquad\forall M\in\mathbb{N}
 \end{equation}
  for a pair (hence for any pair) of nets $(u_\varepsilon)_{\varepsilon\in(0,1]}$ and $(v_\varepsilon)_{\varepsilon\in(0,1]}$ in $\mathcal{E}_{C^1,H^2}([0,T)\times\R^3)$ representing, respectively, $u$ and $v$.

 Existence and uniqueness of the generalised Colombeau $\mathcal{G}_{C^1,H^2}$-solution $u_{\mathrm{Col}}$ to the Cauchy problem for the Hartree equation with $\delta$-coefficient, namely
 \begin{equation}\label{eq:Hartreeversion}
 \begin{cases}
  \;\ii\partial_t u\;=\;-\Delta_x u +\delta\,u+(w*|u|^2)u\,, \\
  u(0,\cdot)\;=\;a\in H^2(\mathbb{R}^3)\,,
 \end{cases}
\end{equation}
 was established in \cite{Dugandzija-Vojnovic-2021} with even-symmetric $w\in W^{2,p}(\mathbb{R}^3,\mathbb{R})$, $p\in(2,\infty]$: in this case, according to the general definition stated above, a representative net $(u_\varepsilon)_{\varepsilon\in(0,1]}$ of such $u_{\mathrm{Col}}$ satisfies
 \begin{equation}\label{eq:epsCol}
 \begin{cases}
  \;\ii\partial_t u_\varepsilon\;=\;-\Delta_x u_\varepsilon +V_\varepsilon u_\varepsilon +(w*|u_\varepsilon|^2)u_\varepsilon \\
  u(0,\cdot)\;=\;a*\rho_\varepsilon
 \end{cases}\qquad \forall\varepsilon\in(0,1]\,,
\end{equation}
 where $V_\varepsilon$ is given by \eqref{eq:general_scaling} with $\eta\equiv 1$, $\sigma =3$, and $\int_{\mathbb{R}^3}V=1$. On the other hand, the classical Cauchy problem
 \begin{equation}\label{eq:HartreeversionCL}
 \begin{cases}
  \;\ii\partial_t u\;=\;-\Delta_x u +(w*|u|^2)u\,, \\
  u(0,\cdot)\;=\;a\in H^2(\mathbb{R}^3)
 \end{cases}
\end{equation}
 is globally (hence also locally) well posed in $H^2(\mathbb{R}^3)$ in the classical (Sobolev) sense \cite[Section 5.3]{cazenave}, say, with solution $u_{\mathrm{cl}}\in C(\mathbb{R},H^2(\mathbb{R}^3))$. Observe that in general
 \[
  u_{\mathrm{Col}}\;=\;\big[(u_\varepsilon)_{\varepsilon\in(0,1]}\big]\;\neq\;[(u_{\mathrm{cl}}*\rho_\varepsilon)_{\varepsilon\in(0,1]}]
 \]
 that is, such two equivalence classes in $\mathcal{G}_{C^1,H^2}([0,T)\times\R^3)$ do not necessarily coincide.

 A weaker identification that replaces the above coincidence in $\mathcal{G}_{C^1,H^2}([0,T)\times\R^3)$, or lack thereof, is the notion of \emph{association}. One says that  $u\in \mathcal{G}_{C^1, H^2}([0,T)\times\R^3)$ is \emph{associated} with a distribution-valued map $v:[0,T)\to\mathcal{D}'(\R^3)$, and then writes $u\approx v$, if there is a representative $(u_\varepsilon)_{\varepsilon\in(0,1]}$ of $u$ such that $u_\varepsilon(t,\cdot)\xrightarrow{\varepsilon\downarrow 0} v(t)$ in $\mathcal{D}'(\R^3)$ for all $t\in[0,T)$.

 A meaningful association one may establish is of the type
 \[
  u_{\mathrm{Col}}\;=\;\big[(u_\varepsilon)_{\varepsilon\in(0,1]}\big]\;\approx\;u_{\mathrm{cl}}\,.
 \]
 In fact, the association above is typically implied by the following notion of 
%  
%  which is obviously implied by the limit $\|u_{\mathrm{cl}}-u_\varepsilon\|_{L^\infty([0,T),L^2(\mathbb{R}^3))}\xrightarrow{\varepsilon\downarrow 0} 0$. This inspires the additional notion
%  
 \emph{compatibility}, formulated as follows for concreteness for the Hartree equation: one says that the Colombeau generalised solution $u_{\mathrm{Col}}=[(u_\varepsilon)_{\varepsilon\in(0,1]}]\in \mathcal{G}_{C^1,H^2}([0,T)\times\R^3)$ to the Cauchy problem \eqref{eq:Hartreeversion}, thus with the $u_\varepsilon$'s satisfying \eqref{eq:epsCol}, is \emph{compatible} with the classical solution $u_{\mathrm{cl}}\in C([0,T),H^2(\mathbb{R}^3))$ to the Cauchy problem \eqref{eq:HartreeversionCL}, precisely when
 \begin{equation}\label{eq:compatibility}
  \lim_{\varepsilon\downarrow 0}\|u_{\mathrm{cl}}-u_\varepsilon \|_{L^\infty([0,T),L^2(\mathbb{R}^3))}\,=\;0\,.
 \end{equation}
 Such a definition does not depend on the representative net for $u_{\mathrm{Col}}$. 
 Clearly, re-interpreting each $u_{\mathrm{cl}}(t,\cdot)$ as a $\mathcal{D}'(\R^3)$-element, the occurrence of \eqref{eq:compatibility} does imply the association $u_{\mathrm{Col}}\approx u_{\mathrm{cl}}$.

 It is worth mentioning that association, compatibility, and variants thereof, between generalised and classical solutions to non-linear PDEs are the object of intensive investigation. This includes, among others, the proof of the analogous of \eqref{eq:compatibility} in $L^\infty([0,T),H^2(\mathbb{R}^3))$ for the semi-linear wave equation \cite{Nedeljkov-Oberguggenberger-Pilipovic-2005}, and in $L^\infty([0,T),\mathcal{D}'(\mathbb{R}^3))$ for hyperbolic conservation laws \cite{Oberguggenberger-Wang-1994}, as well as modified notions such as a type of $L^\infty$-association (thus, stronger than the association introduced above) \cite{Nedeljkov-Pilipovic-1997}.
 
 %a semi-linear hyperbolic system is considered. The concept of $L^\infty$ association of two generalized solutions is defined, which is a stronger concept than association.  Then, $L^\infty$ association between a Colombeau solution and a local continuous solution is proved.
%

%In the limiting case when $\varepsilon\to 0$ we have that $\|u-u*\phi_\varepsilon\|_{H^2}\to 0$. So a very important notion is the notion of \textbf{compatibility}, that is we want that any solution of the regularized equation with initial data $a_\varepsilon$ converges in $H^2$ to $u$. In other words, given $a*\varphi_\varepsilon$ as initial condition (which represents a function $a\in H^2$) we want that the corresponding solution $u_\varepsilon\in\mathcal{E}_{C^1,H^2}$ represents $u\in H^2$, a solution of the original equation.

%Note that if $a\in C^1([0,T), H^\infty)$, then $a$ represents itself and the same holds for the corresponding solution $u\in C^1([0,T), H^\infty)$. Then we automatically have compatibility between the two solutions.
% 
 
 \section{The operator $\Delta_\alpha$}\label{app:Delta-alpha}

 We recall here the standard construction of the singular point-like perturbation $\Delta_\alpha$ of $\Delta$ in dimension $d=3$ -- see, e.g., \cite[Chapter I.1]{albeverio-solvable} and \cite[Section 3]{MO-2016}.

 The operator $S:=(-\Delta)\upharpoonright C^\infty_c(\mathbb{R}^3\setminus\{0\})$ is densely defined, symmetric, and positive with respect to the Hilbert space $L^2(\mathbb{R}^3)$. Its self-adjoint extensions form the one-parameter family $(-\Delta_\alpha)_{\alpha\in\mathbb{R}\cup\{\infty\}}$, where the element corresponding to $\alpha=\infty$ is the Friedrichs extension of $S$ and consists of the self-adjoint negative Laplacian in $L^2(\mathbb{R}^3)$ with domain $H^2(\mathbb{R}^3)$, whereas, for fixed $\alpha\in\mathbb{R}$ and $\lambda>0$,
 \begin{equation}\label{eq:defDelta-alpha}
  \begin{split} 
   \mathrm{dom}(-\Delta_\alpha)\;&=\;\left\{\psi\in L^2(\mathbb{R}^3)\left|
   \begin{array}{c}
    \exists\, f_\psi^{(\lambda)}\!\in H^2(\mathbb{R}^3)\textrm{ such that} \\
    \psi=f_\psi^{(\lambda)}+\frac{\,f_\psi^{(\lambda)}(0)\,}{\alpha+\frac{\sqrt{\lambda}}{4\pi}}\,G_\lambda 
   \end{array}\!\!\right.\right\}, \\
   (-\Delta_\alpha+\lambda\mathbbm{1})\psi\;&=\;(-\Delta+\lambda\mathbbm{1})f_\psi^{(\lambda)}\,,
  \end{split}
 \end{equation}
 having denoted by $G_\lambda$ the fundamental solution (Green function) to $(-\Delta+\lambda 1)G_\alpha=\delta(x)$, explicitly,
 \begin{equation}\label{eq:GlambdaApp}
  G_\lambda(x)\;=\;\frac{\,e^{-|x|\sqrt{\lambda}}\,}{4\pi|x|}\,.
 \end{equation}
 In \eqref{eq:defDelta-alpha} the space $\mathrm{dom}(-\Delta_\alpha)$ is independent of $\lambda$, whereas the decomposition of $\psi$ into a regular $H^2$-part and a singular part proportional to $G_\lambda$ does depends on the chosen $\lambda>0$; each $\psi\in \mathrm{dom}(-\Delta_\alpha)$ and $\lambda>0$ uniquely identifies $f_\psi^{(\lambda)}$.

 Each self-adjoint extension $-\Delta_\alpha$ has spectrum with the following features:
  \begin{equation}
  \begin{split}
   \sigma_{\mathrm{ess}}(-\Delta_\alpha)\;=\;\sigma_{\mathrm{ac}}(-\Delta_\alpha)\;&=\;[0,+\infty)\,, \\
   \sigma_{\mathrm{sc}}(-\Delta_\alpha)\;&=\;\emptyset\,, \\
   \sigma_{\mathrm{p}}(-\Delta_\alpha)\;&=\;
   \begin{cases}
    \quad \;\;\emptyset\,, &\textrm{ if }\alpha\geqslant 0\,, \\
    \{-(4\pi\alpha)^2\}\,,&\textrm{ if }\alpha <0\,.
   \end{cases}
  \end{split}
 \end{equation}
 The negative eigenvalue, when existing, is non-degenerate, and has normalised eigenfunction
 \begin{equation}
  \Psi_\alpha(x)\;:=\;\sqrt{-2|\alpha|\,}\,\frac{\,e^{-4\pi|\alpha||x|}\,}{|x|}\,.
 \end{equation}

 For each $\lambda\in\mathbb{R}^+\setminus\sigma(-\Delta_\alpha)$ the resolvent of $-\Delta_\alpha$ is given by
 \begin{equation}\label{eq:resolventDalpha}
  (-\Delta_\alpha+\lambda\mathbbm{1})^{-1}\;=\;(-\Delta+\lambda\mathbbm{1})^{-1}+\frac{1}{\alpha+\frac{\sqrt{\lambda}}{4\pi}}\,|G_\lambda\rangle\langle G_\lambda|\,.
 \end{equation}
 It is therefore a rank-one perturbation of the free resolvent.
 
 The operator $\Delta_\alpha$ is a perturbation of $\Delta$ in the sole $L^2$-sector of spherically symmetric functions, and coincides with the self-adjoint Laplacian on all the $H^2$-functions orthogonal to the spherically symmetric ones. 
  The $s$-wave scattering length of $-\Delta_\alpha$ amounts to $-(4\pi\alpha)^{-1}$. 
 Beside, $\Delta_\alpha$ is a local operator, in the sense that if $\psi\in\mathrm{dom}(-\Delta_\alpha)$ vanishes on an open $U\subset\mathbb{R}^3$, then $(\Delta_\alpha\psi)|_U\equiv 0$.

 The integral kernel $K_\alpha(x,y;t)$ of the unitary propagator $e^{\ii t\Delta_\alpha}$, $t>0$, generated by $-\Delta_\alpha$ is also explicit \cite{Albeverio_Brzesniak-Dabrowski-1995}, and given by
 \begin{equation}\label{eq:kernelpropagalpha}
 \begin{split}
  K_\alpha(x,y;t)\;=\;
  \begin{cases}
   \begin{array}{c}
    K(x,y;t)+\displaystyle\frac{1}{\,|x|\,|y|\,}\int_0^{+\infty}e^{-4\pi\alpha s}(s+|x|+|y|)\,\times \\
    \times\,K(s+|x|+|y|,0;t)\,\ud s\,,
   \end{array}& \textrm{if }\alpha>0\,, \\ \\
    K(x,y;t)+\displaystyle\frac{2\,\ii\,t}{\,|x|\,|y|\,}\,K(|x|+|y|,0;t)\,,& \textrm{if }\alpha=0\,, \\ \\
    \begin{array}{l}
    K(x,y;t)+e^{\ii t(4\pi\alpha)^2}\Psi_\alpha(x)\Psi_\alpha(y) \\
    +\displaystyle\frac{1}{\,|x|\,|y|\,}\int_0^{+\infty}e^{-4\pi|\alpha| s}(s-|x|-|y|)\,\times \\
    \qquad\qquad\qquad \times\,K(s-|x|-|y|,0,t)\,\ud s\,,
    \end{array}& \textrm{if }\alpha<0\,,
  \end{cases}
 \end{split}
 \end{equation}
 where
 \begin{equation}
  K(x,y;t)\;:=\;\frac{\,e^{-\frac{\:|x-y|^2}{4\ii t}}}{\,(4\pi\ii t)^{\frac{3}{2}}}\,,\qquad t>0
 \end{equation}
  is the kernel of the free Schr\"{o}dinger propagator.

% \bibliographystyle{siam}
% \bibliography{bib_ALE}

\def\cprime{$'$}

\end{document}